\documentclass[12pt]{amsart}

\usepackage[utf8]{inputenc}

\usepackage[margin=1.3in]{geometry}

\usepackage{amsmath}
\usepackage{amsthm}
\usepackage{amssymb}
\usepackage{color}
\input xy
\xyoption{all}
\usepackage[all,cmtip, color]{xy}
\usepackage{tcolorbox}
\usepackage{enumitem, mathdots}
\usepackage{graphicx}
\usepackage[normalem]{ulem}
\usepackage{float}

\theoremstyle{definition}
\newtheorem{theorem}{Theorem}[section]
\newtheorem{lemma}[theorem]{Lemma}
\newtheorem{definition}[theorem]{Definition}
\newtheorem{example}[theorem]{Example}
\newtheorem{c-example}[theorem]{Counter-example}
\newtheorem{Lemma}[theorem]{Lemma}
\newtheorem{corollary}[theorem]{Corollary}
\newtheorem{Prop}[theorem]{Proposition}
\newtheorem{setup}[theorem]{Setup}
\newtheorem{remark}[theorem]{Remark}

\numberwithin{equation}{section}

\usepackage{hyperref}
\usepackage[pagewise]{lineno}

\newenvironment{psmallmatrix}
  {\left(\begin{smallmatrix}}
  {\end{smallmatrix}\right)}
\newcommand{\Cal}[1]{{\mathcal #1}}

\newcommand{\paral}[1]{\ar@<0.3ex>[#1] \ar@<-0.3ex>[#1]}

\usepackage{dynkin-diagrams}

\newcommand{\gen}{\mathsf{Gen}}
\newcommand{\cogen}{\mathsf{Cogen}}
\newcommand{\add}{\mathsf{add}}
\newcommand{\tors}{\mathsf{tors}}

\DeclareMathOperator{\Hom}{Hom}

\DeclareMathOperator{\Triv}{\mathsf{Triv}}
\DeclareMathOperator{\im}{Im}
\DeclareMathOperator{\rad}{rad}

\DeclareMathAlphabet\mathbfcal{OMS}{cmsy}{b}{n}

\renewcommand{\mod}{\mathsf{mod}}

\newdir{ >}{{}*!/-7pt/@{>}}

\title{Lattices of pretorsion classes}

\author[F. Campanini]{Federico Campanini}
\address{Universit\'e catholique de Louvain, Institut de Recherche en Math\'ematique et Physique, 1348 Louvain-la-Neuve, Belgium}
\email{federico.campanini@uclouvain.be}

\author[F. Fedele]{Francesca Fedele}
\address{School of Mathematics, University of Leeds, Leeds, LS2 9JT, United Kingdom}
\email{f.fedele@leeds.ac.uk}

\author[E. Y\i ld\i r\i m]{Emine Y\i ld\i r\i m}
\address{International Center for Mathematical Sciences, Institute of Mathematics and
Informatics, Bulgarian Academy of Sciences, Acad. G. Bonchev Str., Bl. 8, Sofia
1113, Bulgaria}
\email{e.yildirim@math.bas.bg}


\begin{document}

\begin{abstract}
Since their introduction, torsion theories have played a key role in the study of abelian and pointed categories. In representation theory, torsion theories and lattices of torsion classes of $\mod A$, for $A$ a finite-dimensional algebra, have been widely studied.
The more recent definition of pretorsion theories, that can be given for any category, has expanded the theory, giving many more instances of ``non-pointed torsion theories'' in unexpected settings.
In this work, we introduce and study the lattice $\Cal L_t(A)$ of pretorsion classes of $\mod A$. These lattices are in close connection with the lattices $\tors A$ of torsion classes of $\mod A$. We fully describe the completely join-irreducible elements of $\Cal L_t(A)$. Moreover, we characterise and give a full classification of when $\Cal L_t(A)$ is distributive and further describe when it can be identified with the \emph{distributive closure} of $\tors A$.
Finally, we show how the lattices of pretorsion classes, together with their duals, can be used to build pretorsion theories in $\mod A$. 
\end{abstract}
\maketitle

\tableofcontents

\section{Introduction}

Starting from Dickson's~\cite{D} categorification of the classic construction of ``breaking'' an abelian group into its torsion subgroup and torsion-free quotient group, \emph{torsion theories} have been a key concept to study abelian categories by breaking them into two subcategories: a \emph{torsion} and a \emph{torsion-free class}.
In recent years, lattices of torsion classes of module categories have been extensively studied and have gained importance. The aim of this paper is to introduce and study lattices of pretorsion classes: a recent, non-pointed version of torsion classes.  Throughout the paper, we will work in the following setup.
 
\begin{setup}
Let $k$ be an algebraically closed field and $A$ be a finite-dimensional, unitary, associative $k$-algebra. The (abelian) category of finitely generated right $A$-modules is denoted by $\mod A$. All subcategories are assumed to be full and closed under isomorphisms (replete).
\end{setup}

\subsection{Classic case: lattices of torsion classes}
In modern representation theory, lattices of torsion classes of $\mod A$~\cite{IT, T, DIRRT} have brought a new look to the topic. Understanding their poset structure led to many combinatorial and geometric discoveries, advancing the understanding of torsion theories.

The systematic study of how subcategories interact via lattice-theoretic structures started with the work of Ingalls and Thomas~\cite{IT} on the lattice of noncrossing partitions, through representation theory of quivers. They proved extraordinary bijections between combinatorial objects and many important classes of subcategories such as torsion classes, wide subcategories and cluster-tilting objects. Subsequently, the foundational work of Adachi, Iyama and Reiten~\cite{AIR} on $\tau$-tilting theory proved that the set of functorially finite torsion classes of $\mod A$ forms a lattice.

Demonet, Iyama, Reading, Reiten and Thomas~\cite{DIRRT} proved the set of torsion classes $\tors A$ of $\mod A$, ordered by inclusion, forms a \emph{complete lattice} and showed many interesting properties of the associated Hasse diagram.
Further important results on these lattices were proven by Barnard, Carroll and Zhu in~\cite{BCZ}. In particular, they showed that the \emph{completely join-irreducible elements} of $\tors A$ are in bijection with the isomorphism classes of certain indecomposable modules in $\mod A$, called \emph{bricks}. Moreover, in~\cite{DIRRT} the authors studied lattice congruences and lattice quotients of $\tors A$ and proved isomorphisms between Reading's \emph{Cambrian lattices}~\cite{R,R06} (that are lattice quotients of the weak order of finite Coxeter groups) and $\tors (kQ)$, when $Q$ is a Dynkin quiver. For type $A$, these lattices, known as \emph{Tamari lattices}, can be associated to triangulations of disks, revealing their intrinsic combinatorics and connection to cluster algebras.

Notice that the lattice $\tors A$ is \emph{distributive} only in a trivial situation: if $A$ is connected, this happens only when $\tors A$ consists of the zero subcategory and $\mod A$, see Remark~\ref{rem_distributive_tors}. However, using Jasso's reduction~\cite{Ja}, Reading, Speyer and Thomas~\cite{RST} proved that $\tors A$ always has the important, less restrictive property of being \emph{semidistributive}. Jasso's reduction techniques were further used by Asai and Pfeiffer to study lattices of torsion classes of certain abelian subcategories of abelian length categories, such as $\mod A$, in~\cite{AP}.

The field of lattice theory in the context of representation theory is enourmously rich, connecting to geometry g-vector fans, $\tau$-tilting fan, scattering diagrams, and many other recent discoveries in mathematics and physics. 
Moreover research on lattices of subcategories continues and is a very active research topic beyond abelian categories: a testimony are Krause's~\cite{K, K23} results on lattices of thick subcategories in triangulated categories.

\subsection{This paper: lattices of pretorsion classes}
Since their introduction in~\cite{D}, torsion theories have been widely studied both in abelian and general pointed categories, and extended in many different directions.  Recently, pretorsion theories were defined, by Facchini, Finocchiaro and Gran~\cite{FF, FFG}, for any category as ``non-pointed torsion theories'': the zero object and zero morphisms are replaced by a class of ``trivial objects'' and an ideal of ``trivial morphisms'', respectively (see Definition~\ref{defn_pretorsion}). Pretorsion theories appear in several different contexts, such as topological spaces and topological groups, internal preorders, categories, preordered groups, V-groups, crossed modules, etc. See for instance \cite{FFG, BCGT} and references therein.

It is natural to analyse pretorsion theories in abelian categories in depth. Two authors of the current paper, Campanini and Fedele, carried out the first research in this direction in their recent work~\cite{CF}. Their results present ways of obtaining pretorsion theories starting from the better understood torsion theories, providing many new examples of pretorsion theories in algebraic settings.

Motivated by the developments in lattice theory of torsion classes, in this paper we aim to investigate lattices of pretorsion classes and develop this theory in parallel with the well-established one of torsion classes. 

A pretorsion class in $\mod A$ is a subcategory closed under quotients and direct sums, see Proposition~\ref{prop_noetherian}. Hence, pretorsion classes in $\mod A$ are subcategories of the form $\gen (\Cal S)$ for some class $\Cal S$ of objects in $\mod A$. In Proposition~\ref{prop_complete_lattice_torsionclasses}, we show that the set of pretorsion classes in $\mod A$, ordered by inclusion, forms a complete lattice, which we denote by $\Cal L_t(A)$.

Once the lattice $\Cal L_t(A)$ is established, we proceed to prove some of its important properties. Firstly, we describe its completely join-irreducible elements, proving an analogue of the above mentioned classic result~\cite[Theorem~1.5]{BCZ}. 

\vspace{2mm}
\noindent\textbf{Theorem~\ref{thm_joinirr}.}
\textit{
There is a bijection between the isomorphism classes of indecomposable modules in $\mod A$ and the completely join-irreducible elements in the lattice $\Cal L_t (A)$, given by $M\mapsto \gen (M)$.
}
\vspace{2mm}

As we recalled above, the lattice $\tors A$ is hardly ever distributive. Surprisingly, there are many more instances of the lattice of pretorsion classes being distributive. We give a characterisation of when this happens.

\vspace{2mm}
\noindent\textbf{Theorem~\ref{thm_distributive-characterisation} and Proposition~\ref{prop_iff_unique_max}.}
\textit{
The following are equivalent:
\begin{itemize}
    \item the lattice $\Cal L_t(A)$ of pretorsion classes is distributive,
    \item $\add(\Cal T_1 \cup \Cal T_2) = \gen( \Cal T_1 \cup \Cal T_2 )$ for every pair of pretorsion classes $\Cal T_1,\Cal T_2\in\Cal L_t (A)$,
    \item each indecomposable module in $\mod A$ has a unique maximal submodule.
\end{itemize}
}
\vspace{2mm}

As we recall in Remark~\ref{rem_equivalence_modquiv}, given an algebra $A$ as in our setting, one can construct a bound quiver algebra $kQ/\Cal I$ such that $\mod A\cong \mod (kQ/\Cal I)$ as categories, where $Q$ is a finite quiver and $\Cal I$ an admissible ideal of the path algebra $kQ$. In order to classify when $\mod A$ has distributive lattice of pretorsion classes, it is then enough to do so for bound quiver algebras. We do this in the following result.

\vspace{2mm}
\noindent\textbf{Theorem~\ref{thm_classification_distributive}.}
\textit{
Let $A=kQ/\Cal I$ be a bound quiver algebra. Then $\Cal L_t(A)$ is distributive if and only if each vertex in $Q$ has at most one arrow ending at it and two going out of it, and when there are precisely one entering arrow $\alpha$ and two exiting arrows $\beta$ and $\gamma$, then $\alpha\beta$ or $\alpha\gamma$ is in $\Cal I$.
}
\vspace{2mm}

As we point out in Remark~\ref{remark_distributive_finite}, algebras as in Theorem~\ref{thm_classification_distributive} all are of finite-representation type. Hence when $\Cal L_t(A)$ is distributive, it is a finite lattice.

It is natural to explore the connection between $\tors A$ and $\Cal L_t (A)$ for a given $A$. Since torsion classes are a special kind of pretorsion classes, $\tors A$ is a subposet, but usually not a sublattice, of $\Cal L_t (A)$. 
As mentioned earlier, the completely join-irreducibles of $\tors A$ are in bijection with isomorphism classes of certain indecomposable modules (bricks), while the completely join-irreducibles of $\Cal L_t(A)$ are in bijection with isomorphism classes of all of the indecomposable modules in $\mod A$.
A natural question to ask is what is the relationship between $\tors A$ and $\Cal L_t (A)$ when all of the indecomposable modules of $A$ are bricks. The algebras for which this is the case are called \emph{locally representation directed} and are known to be of finite-representation type, so that their corresponding lattice of pretorsion classes is finite. 

We define the \emph{distributive closure} of a finite lattice $\Cal L$ to be the (distributive) lattice $\Cal L^*$ of order ideals of the poset of join-irreducibles of $\Cal L$, ordered under inclusion. 

\vspace{2mm}
\noindent\textbf{Theorem~\ref{thm_join_irred_coincide} and Proposition~\ref{prop_distr_closure}.}
\textit{
The join-irreducibles of $\Cal L_t(A)$ coincide with the join-irreducibles of $\tors A$ if and only if $A$ is a locally representation directed algebra.
Moreover, $(\tors A)^*$ can be identified with $\Cal L_t (A)$ if and only if $\Cal L_t(A)$ is distributive and $A$ is locally representation directed.
}
\vspace{2mm}

 Note that in view of Remark~\ref{rem_equivalence_modquiv}, to fully describe when $(\tors A)^*$ can be identified with $\Cal L_t (A)$, it is again enough to study the case when $A$ is a bound quiver algebra. In the extended version of Proposition~\ref{prop_distr_closure}, we explicitly describe the bound quiver algebras for which this is true.

Finally, given an algebra $A$, the lattice of pretorsion classes $\Cal L_t(A)$ and the lattice of pretorsion-free classes $\Cal L_{tf}(A)$ can be used to build all the pretorsion theories in $\mod A$. We show some methods on how to do so in Proposition~\ref{prop_build_pret_fromlattices}.

\medskip
The paper is organised as follows.
In Section~\ref{sec:background} we recall the relevant background on pretorsion classes, lattices and $\mod A$. In Section~\ref{sec:latticePT} we introduce the lattice $\Cal L_t(A)$ of pretorsion classes of $\mod A$ and describe its completely join-irreducibles.
In Section~\ref{sec:dist}, we study and classify the distributivity of $\Cal L_t(A)$, we conclude the section with a a study of the relationship between $\tors A$ and $\Cal L_t(A)$.
Finally, the Appendix gives methods to construct pretorsion theories using the lattices and also contains some examples: a description of all pretorsion theories in $\mod(k\mathbb{A}_2)$ and the explicit construction of $\Cal L_t(k\mathbb{A}_3)$ for the three possible orientations of $\mathbb{A}_3$.

\subsection*{Acknowledgements} F.C. is a postdoctoral researcher of the Fonds de la Recherche Scientifique - FNRS, F.F. is supported by the EPSRC Programme Grant\\
EP/W007509/1 and E.Y. by the Simons Foundation, grant SFI-MPS-T-Institutes-00007697, and the Ministry of Education and Science of the Republic of Bulgaria, grant DO1-239/10.12.2024. 

The authors would like to thank Ibrahim Assem and Iacopo Nonis for helpful discussions, and Aslak Buan and Bethany Marsh for the discussion on the proof of Lemma~\ref{lemma_indec_ess_gen}. E.Y. thanks Banff International Research Station and the organisers of the 2025 workshop on Lattice Theory as this wonderful workshop inspired this paper.

\section{Background}\label{sec:background}

The aim of the paper is the study of lattices of pretorsion classes for $\mod{A}$ for $A$ a finite-dimensional $k$-algebra. We aim to address a broad class of mathematicians, including algebraists, combinatorialists and both representation and category theorists. Hence, in this section we recall the needed background in these topics:
we first recall the definition and some results on pretorsion theories, then some basics on lattice theory, and finally we give an overview on finitely generated modules over finite-dimensional algebras.

\subsection{Pretorsion classes and monocoreflective subcategories}
Pretorsion theories were firstly defined in \cite{FF, FFG}, as follows.
Let $\Cal C$ be an arbitrary category and fix a class $\Cal Z$ of objects of $\Cal C$, that we shall call \emph{the class of trivial objects}. A morphism $f\colon A\to A'$ in $\Cal C$ is \textit{$\Cal Z$-trivial} if it factors through an object of $\Cal Z$. Given any two objects $X$ and $Y$, we denote by $\Triv(X,Y)$ the class of $\Cal Z$-trivial morphisms from $X$ to $Y$ and by $\Triv$ the class of all $\Cal Z$-trivial morphisms in $\Cal C$. Notice that $\Triv$ is an ideal of morphisms, that is, for every pair of composable morphisms $f$ and $g$ in $\Cal C$, $fg\in \Triv$ whenever $f$ or $g$ is in $\Triv$. Hence, it is possible to consider the notions of $\mathcal Z$-kernel and $\mathcal Z$-cokernel, defined by replacing, in the definition of kernel and cokernel, the ideal of zero morphisms with the ideal of trivial morphisms induced by the class $\mathcal Z$ as follows.

\begin{definition}
A morphism $\varepsilon\colon X\to A$ in $\Cal C $ is a \emph{$\Cal Z$-kernel} of $f\colon A \to A'$ if $f\varepsilon$ is a $\Cal Z$-trivial morphism and, whenever $\lambda \colon Y\to A$ is a morphism in $\Cal C$ and $f\lambda$ is $\Cal Z$-trivial, there exists a unique morphism $\lambda'\colon Y\to X$ in $\Cal C$ such that $\lambda=\varepsilon\lambda'$.
The notion of \emph{$\Cal Z$-cokernel} is defined dually. A sequence $A\overset{f}{\to}B\overset{g}{\to}C$ is called a \emph{short $\Cal Z$-exact sequence} if $f$ is a $\Cal Z$-kernel of $g$ and $g$ is a $\Cal Z$-cokernel of $f$.
\end{definition}

It can be easily seen that $\Cal Z$-kernels and $\Cal Z$-cokernels, whenever they exist, are unique up to isomorphism and they are monomorphisms and epimorphisms respectively \cite{FFG}.

\begin{definition}\label{defn_pretorsion}
Let $\Cal T$ and $\Cal F$ be (full and replete) subcategories of $\Cal C$.
We say that the pair $(\Cal T,\Cal F)$ is a \emph{pretorsion theory} in $\Cal C$ with class of trivial objects $\Cal Z:=\Cal T\cap \Cal F$, if the following two properties are satisfied:
\begin{itemize}
\item
$\Hom(T, F)=\Triv(T,F)$, for every $T \in \Cal T$ and $F \in \Cal F$;
\item
for every object $X$ of $\Cal C$ there is a short $\Cal Z$-exact sequence
$$\xymatrix{ T_X \ar[r]^f &  X \ar[r]^g &  F_X}$$ with $T_X\in\Cal T$ and $F_X\in\Cal F$.
\end{itemize}

If $(\Cal T, \Cal F)$ is a pretorsion theory in $\Cal C$, we call the subcategory $\Cal T$ a {\em pretorsion class} and the subcategory $\Cal F$ a {\em pretorsion-free class}.
\end{definition}

\begin{remark}
    When $\Cal C$ is pointed and $\Cal T\cap \Cal F=0$, we recover the usual notions of torsion theory. In particular, any torsion class (resp. torsion-free class) is a pretorsion class (resp. pretorsion-free class).
\end{remark}

The notion of pretorsion class can also be introduced without the use of pretorsion theories, as shown in the definitions and in the lemma below.

\begin{definition}\label{defn_coreflective}
    A subcategory $\Cal{A}$ of a category $\Cal{C}$ is called {\em coreflective} if the full embedding $E_\Cal A\colon \Cal A \to \Cal C$ admits a right adjoint $R_\Cal{A}:\Cal{C}\to\Cal{A}$. Moreover, if all the counit components are monomorphisms, $\Cal{A}$ is called {\em monocoreflective}.
\end{definition}

Note that coreflective subcategories are equivalent to strongly covering subcategories as defined below, see for example \cite[dual of Lemma~4.3]{HJV}.

\begin{definition}
    A subcategory $\Cal{A}$ of a category $\Cal{C}$ is called {\em strongly covering} if  each object $X\in \Cal{C}$ has a {\em strong $\Cal{A}$-cover}, that is a morphism $\alpha: A\to X$ such that $A\in\Cal{A}$ and any morphism $\alpha':A'\to X$ with $A'\in \Cal{A}$ uniquely factors through $\alpha$.
\end{definition}

A strong $\Cal{A}$-cover of an object $X \in \Cal C$ is unique up to isomorphism and it corresponds to the $X$-component of the counit of the adjunction $R_\Cal A \vdash E_\Cal A$ (as in Definition \ref{defn_coreflective}). In particular, monocoreflective subcategories coincide with strongly covering subcategories where all strong covers are monomorphisms.

The remaining results in this subsection are somewhat well-known, but we include them and their proofs for the convenience of the reader.

\begin{lemma}\label{lemma_monocor_pretorsion}
    Let $\Cal C$ be a category. The pretorsion classes of $\Cal C$ are precisely the monocoreflective subcategories of $\Cal C$.
\end{lemma}

\begin{proof}
    Assume that $(\Cal T, \Cal F)$ is a pretorsion theory. By \cite{FFG}, the ``torsion functor" $T\colon \Cal C\to \Cal T$ is the right adjoint of the full embedding $E_\Cal T\colon \Cal T \to \Cal C$. Moreover, all the counit components are monomorphisms. On the other hand, if $\Cal T$ is a monocoreflective subcategory of $\Cal C$, then it is immediate to check that the pair $(\Cal T, \Cal C)$ is a pretorsion theory in $\Cal C$.
\end{proof}

It is well-known that, in suitable frameworks, monocoreflective subcategories can be characterised in terms of closure under quotients and coproducts. Here we include a variation of the statement tailored for our purposes, including a sketch of the proof. 

\begin{definition}\label{defn_wellp_noet}
    A category $\Cal C$ is said to be {\em well-powered} if for every object $X \in \Cal C$, the subobjects of $X$ form a set.
    A category $\Cal C$ is said to be {\em Noetherian}, if for every object $X \in \Cal C$, any ascending chain of subobjects of $X$ stabilizes.
\end{definition}

\begin{Prop}\label{prop_noetherian}
    Let $\Cal C$ be a well-powered Noetherian abelian category. Then, a subcategory $\Cal T$ of $\Cal C$ is monocoreflective if and only if $\Cal T$ is closed under quotients and existing coproducts.
\end{Prop}

\begin{proof}
    Assume that $\Cal T$ is monocoreflective in $\Cal C$. Then, $\Cal T$ is closed under all colimits existing in $\Cal C$ (\cite[Section 3.5]{Borceux1}). Moreover, if $f\colon T \to X$ is an epimorphism with $T \in    \Cal T$, then $f$ factors through the strong $\Cal T$-cover $\varepsilon_X\colon T_X\to X$ of $X$. Thus, $\varepsilon_X$ is an isomorphism and $X \in \Cal T$.
    
    Conversely, assume that $\Cal T$ is closed under quotients and existing coproducts and fix $X \in \Cal C$. We want to find a strong $\Cal T$-cover for $X$. Consider the set $\Gamma$ of all subobjects $T$ of $X$ such that $T \in \Cal T$. Since $\Cal C$ is Noetherian, every chain in $\Gamma$ has an upper bound, hence, by Zorn's Lemma, there exists a maximal element $T_X \in \Gamma$. Recall that the join of two subobjects $A,B$ of $X$ is given by the image of the canonical morphisms $A\oplus B \to X$ and if $A,B \in \Cal T$, then their join is in $\Cal T$ by the hypothesis on $\Cal T$. From this, we get that $T_X$ is in fact a maximum in $\Gamma$. It is then routine to show that the inclusion $T_X \hookrightarrow X$ is a (mono) strong $\Cal T$-cover of $X$.
\end{proof}

\begin{remark}
    Proposition~\ref{prop_noetherian} shows that if $\Cal C$ is a well-powered Noetherian abelian category, then its pretorsion classes are exactly its subcategories closed under quotients and existing coproducts.
    Notice this agrees with Stenstr\"{o}m's earlier definition of \emph{pretorsion class} \cite[pp.~137]{Sten} in the setting of complete, cocomplete, locally small, abelian categories.
\end{remark}

\subsection{Lattices}

A \emph{lattice} $(\Cal L, \vee, \wedge)$ is a poset (a partially ordered set) $\Cal L$ with two operations:  \emph{join} and \emph{meet}. For every pair $x,y$ of elements, the join, denoted $x\vee y$, is the least upper bound and the meet, denoted $x\wedge y$, is the greatest lower bound. A lattice is \emph{distributive} if these two operations distribute over each other, that is for any three elements $x,y,z$ in $\Cal L$, we have that

\[
x \wedge (y \vee z) = (x \wedge y) \vee (x \wedge z),
\]
or equivalently
\[
x \vee (y \wedge z) = (x \vee y) \wedge (x \vee z).
\]

A lattice $\Cal L$ is said to be \emph{complete} if any subset $S$ of $\Cal L$ has a unique least upper bound and a unique greatest lower bound, denoted respectively by
$$
\bigvee_{x\in S} x \qquad \text{ and } \qquad \bigwedge_{x\in S} x.
$$
A complete lattice has a \emph{bottom element} (the meet of all elements in $L$) and a \emph{top element} (the join of all elements in $L$). Note that a finite lattice is clearly complete.

The interested reader can find more on general lattice theory in~\cite{G, Stan}. Moreover, lattices in the context of representation theory of algebras have received quite popularity with cornerstone papers such as~\cite{AIR,DIRRT,K}.

\begin{definition}
    Let $(\Cal L,\vee, \wedge)$ be a lattice. An element $x \in \Cal L$ different from the bottom element is called a {\em join-irreducible} if whenever $x=a\vee b$, either $x=a$ or $x=b$. An element $x$ different from the bottom element is called a {\em completely join-irreducible} if there exists $x_0 \in \Cal L$, $x\neq x_0$, such that for every $a \in \Cal L$, $a< x$ if and only if $a\leq x_0$. If the lattice is complete, the latter is equivalent to requiring that  $\bigvee_{a<x}a<x$. A completely join-irreducible element is always join-irreducible and the two notions coincide if $\Cal L$ is a finite (complete) distributive lattice.
\end{definition}

\subsection{Finitely generated modules over finite-dimensional \texorpdfstring{$k$}{k}-algebras}

The main aim of this paper is studying and classifying pretorsion classes for the category $\mod{A}$ of finitely generated right modules over a finite-dimensional $k$-algebra $A$, where we fix $k$ to be an algebraically closed field. These kind of categories are very well-known and studied, see for example  \cite{ASS} for more details.

Notice that $\mod A$ is an abelian category which is also well-powered and Noetherian in the sense of Definition~\ref{defn_wellp_noet}. Hence its pretorsion classes can be described using Proposition~\ref{prop_noetherian} as its subcategories closed under quotients and finite direct sums. Moreover, $\mod A$ is a Krull-Schmidt category in the following sense.

\begin{definition}
A category $\Cal C$ is said to be a {\em Krull-Schmidt category} if 
\begin{enumerate}
    \item
    every object in $\Cal C$ is isomorphic to the direct-sum of finitely many indecomposable objects,
    \item
    whenever $M_1\oplus \cdots\oplus M_r\cong N_1\oplus \cdots \oplus N_s \in \Cal C$ with $M_1,\dots,M_r,N_1,\dots,N_s$ indecomposable objects, then $r=s$ and there exists a permutation $\sigma \in S_r$ such that $M_i \cong N_{\sigma(i)}$ for every $i=1,\dots,r$.
\end{enumerate}
\end{definition}

Thanks to the above property, it is enough to  understand the indecomposable modules in $\mod A$ in order to study all of the objects of the category. 

We recall here an important class of algebras.

\begin{definition}[{\cite{D91}}]
    An $A$-module $M$ is a {\em brick} if its endomorphism ring is a division ring, or equivalently, if $\operatorname{End}_A(M)\cong k$.

    A $k$-algebra $A$ such that each indecomposable module $\mod{A}$ is a brick, is called a {\em locally representation directed algebra}.
\end{definition}

Locally representation directed algebras are {\em representation-finite}, that is, there are only finitely many isomorphism classes of indecomposable modules in $\mod A$. Examples of locally representation directed algebras are representation directed algebras such as representation-finite hereditary algebras and tilted algebras (see \cite[Chapter IX]{ASS}).

For most of the paper we will study $\mod A$ for $A$ a \emph{bound quiver algebra}. This will allow us to obtain a complete classification of when the lattices of pretorsion classes of $\mod A$ are distributive in terms of the associated quiver, see Section~\ref{section_classification_distr}.

We quickly recall the definition of such algebras, for more details and basic properties, see \cite{ASS} or \cite{S}.
Let $Q$ be a \emph{finite quiver}, that is a directed finite multigraph. Its \emph{path algebra} $kQ$ is the $k$-algebra with basis the paths in $Q$ (including the trivial path at each vertex) and multiplication given by concatenation of paths when possible and zero otherwise. In order to guarantee that these algebras are finite-dimensional, we restrict ourselves to bound quiver algebras, that is we factor $kQ$ by an \emph{admissible ideal}, which is an ideal of $kQ$ that contains some power of any present oriented cycle and does not contain any arrow.

We write $A=kQ/\Cal I$ to mean a bound quiver algebra with $Q$ a finite quiver and $\Cal I$ an admissible ideal, so that $A$ is a (unitary, associative) finite-dimensional $k$-algebra. 
Recall that an algebra $A$ is \emph{connected} if it is not the direct product of two algebras. Note that this is equivalent to assuming  that $Q$ is \emph{connected}, that is the underlying unoriented graph is connected.

By the following remark, in order to study $\mod A$ for finite-dimensional $k$-algebras, it is enough to study $\mod (kQ/\Cal I)$ for bound quiver algebras $kQ/\Cal I$.

\begin{remark}\label{rem_equivalence_modquiv}
    Let $A$ be a finite-dimensional $k$-algebra, where $k$ is an algebraically closed field. Then there exists a quiver $Q$ and an admissible ideal $\Cal I$ of $kQ$ such that there is an equivalence of categories $\mod A\cong \mod (kQ/\Cal I)$, see~\cite[Chapter~II]{ASS}.
    
    More precisely, there exists a basic algebra $A^b$ and an equivalence of categories $\mod A\cong \mod A^b$. Moreover, one can construct the \emph{ordinary quiver} $Q_{A^b}$ of the basic algebra $A^b$, and $A^b\cong kQ_{A^b}/\Cal I$ for some admissible ideal $\Cal I$ of $kQ_{A^b}$.
    Hence, we have an equivalence of categories $\mod A\cong \mod (kQ_{A^b}/\Cal I)$. As already said, since in this paper we are interested in the study of such module categories rather than the algebra itself, we may assume $A$ is a bound quiver algebra from the beginning. In this paper, modules over $A$ are identified with \emph{quiver representations} over the bound quiver, in particular we use the notation from~\cite[Remark 1.1]{S} for indecomposable modules.
\end{remark}

\subsubsection{The lattice of torsion classes of $A$}\label{subsection_torsA}
We conclude this section with a brief overview of the lattice $\tors A$ of torsion classes of $\mod A$.
Later in the paper, we will compare $\tors A$ to the lattice of pretorsion classes of $\mod A$.

By a \emph{torsion class} we mean a subcategory that appears as the first half of a torsion theory in the ambient (pointed) category.
Recall that since $\mod A$ is a well-powered Noetherian abelian category, it is well-known that its torsion classes are exactly the subcategories of $\mod A$ that are closed under quotients, finite direct sums and extensions. For a class $\Cal S$ of modules in $\mod A$, we denote by $\Cal T(\Cal S)$ the smallest torsion class in $\mod A$ containing $\Cal S$.

We denote by $\tors A$ the set of torsion classes of $\mod A$. Note that this is a poset ordered by inclusion. As widely studied recently, $\tors A$ is a complete lattice with many interesting properties, see~\cite{AP, BCZ, DIRRT, GM, Ja, K} and the survey paper~\cite{T}.

An important feature of $\tors A$ is that its completely join-irreducibles are in bijection with the isomorphism classes of the bricks in $\mod A$, see~\cite[Theorem~3.3]{DIRRT} and~\cite[Theorem~1.5]{BCZ}. In Section~\ref{section_joinirred} we will see a similar result for the lattice of pretorsion classes of $\mod A$.

Moreover, $\tors A$ is completely semidistributive, see~\cite[Theorem~3.1]{DIRRT} (note that semidistributivity was first proved in~\cite{GM}), and further properties following this feature were studied in~\cite{T}. 
On the other hand, putting together the results on this lattice, one can see that $\tors A$ is very rarely distributive, as we show in the following remark.

\begin{remark}\label{rem_distributive_tors}
    Note that in order to determine when $\tors A$ is distributive, it is enough to study connected algebras. In fact, by an argument analogue to the one of Proposition~\ref{remark_directprod_distributive}, if $A=\prod_{i=1}^m A_i$ is a finite direct product of algebras, then
    $$
    \tors A\cong \tors A_1\times \dots\times \tors A_m
    $$
    as lattices and $\tors A$ is distributive if and only if $\tors A_i$ is distributive for each $i=1,\dots, m$.

    Moreover, in view of Remark~\ref{rem_equivalence_modquiv}, we may assume that $A=kQ/\Cal I$ is a bound quiver algebra and $Q$ is a connected quiver.
    If $Q$ only has a single vertex (and possibly many loops at it) then there is only one brick in $\mod A$ up to isomorphism: the simple module. Then, by \cite[Theorem~1.5]{BCZ}, we conclude that $\tors A$ only has one completely-join irreducible and the only torsion classes are $\mod A$ and $0$. In this trivial case, the lattice $\tors A$ is distributive.

    Suppose now that $Q$ has at least two vertices. Since it is connected, there is at least one arrow $\alpha: i\to j$ in $Q$, where $i$ and $j$ are distinct vertices. Let $S_i$ and $S_j$ be the simple modules corresponding to the vertices $i$ and $j$ respectively, and $M$ be the indecomposable representation with a copy of $k$ at vertices $i$ and $j$, the identity map at $\alpha$, and zero vector spaces and zero maps elsewhere.
    There is a short exact sequence of the form $0\to S_j\to M\to S_i\to 0$ and hence, in $\tors A$ we have $\Cal T(M)\subseteq\Cal T(S_j)\vee \Cal T(S_i)$ and $\Cal T(S_i)\subsetneq \Cal T(M)$, where the latter is a strict inclusion by \cite[Theorem~1.5]{BCZ} as $M$ and $S_i$ are non-isomorphic bricks. Applying \cite[Lemma~7.2]{T}, it is easy to see that $S_j\not\in \Cal T(M)$. Moreover, if $S$ is simple, using the filtration description of torsion classes from \cite[Section~2]{T}, one can see that the only torsion classes contained in $\Cal T(S)$ are $0$ and $\Cal T(S)$, and so $\Cal T(S_i)\wedge\Cal T (S_j)=0$ and $\Cal T(M)\wedge \Cal T (S_j)=0$. Hence
    $$
    \Cal T(M)\wedge (\Cal T (S_i)\vee \Cal T(S_j))=\Cal T(M)\neq \Cal T(S_i)= (\Cal T(M)\vee \Cal T (S_i))\wedge( \Cal T (M)\wedge \Cal T(S_j))
    $$
    and $\tors A$ is not distributive.

    In conclusion, allowing $A$ also not to be connected,  $\tors A$ is distributive if and only if it is isomorphic to a lattice product of copies of the lattice 
    $$\xymatrix@R=1em{
    y \ar@{-}[d]\\ x.
    }$$
    In terms of algebras, this happens if and only if $A\cong \prod_{i=1}^m A_i$, with $A_1,\dots,A_m$ connected and such that $\mod A_i\cong \mod (kQ_i/\Cal I)$ with $Q_i$ having a single vertex, for each $i=1,\dots,m$.
\end{remark}

We will see in Section~\ref{section_classification_distr} that the lattice of pretorsion classes of $\mod A$ is much more rich in this sense: there are many nontrivial cases when the lattice is distributive. We give a full classification of when this happens, showing also when it coincides with the distributive closure of $\tors A$.

\section{Lattices of pretorsion classes}~\label{sec:latticePT}

Recall that in this paper $A$ is a finite-dimensional algebra over an algebraically closed field $k$ and $\mod A$ denotes the category of finitely generated (right) modules over $A$. Recall also that $\mod A$ is a well-powered Krull-Schmidt Noetherian abelian category.

\subsection{The lattice \texorpdfstring{$\Cal L_t(A)$}{LtA} of pretorsion classes}
\begin{Lemma}
    If $\{\Cal{T}_i\}_{i \in I}$ is a family of pretorsion classes in $\mod A$, then $\bigcap_{i\in I}\Cal{T}_i$ is a pretorsion class.
\end{Lemma}

\begin{proof}
    By Proposition~\ref{prop_noetherian}, it is enough to show that $\bigcap_{i\in I}\Cal{T}_i$ is closed under quotients and existing coproducts. But this is clear, since all the $\Cal T_i$'s have these properties.
\end{proof}

We can thus give the following definition.

\begin{definition}\label{defn_Gen}
    Let $\Cal S$ be a class of objects of $\mod A$. We denote by $\gen(\Cal S)$ the smallest pretorsion class containing $\Cal S$.
\end{definition}

Thanks to Proposition~\ref{prop_noetherian}, $\gen(\Cal S)$ is the smallest (full) subcategory of $\mod A$ closed under quotients and finite direct sums, that is, the subcategory of $\mod A$ whose objects are quotients of finite direct sums of modules in $\Cal S$. Thus Definition~\ref{defn_Gen} coincides with the one given, for instance, in \cite[Chapter~VI]{ASS}.

\begin{definition}
   We denote by $\Cal L_t(A)$ the partially ordered set of pretorsion classes of $\mod A$, ordered under inclusion.
\end{definition}

\begin{Prop}\label{prop_complete_lattice_torsionclasses}
    The poset $\Cal L_t(A)$ is a complete lattice, with meet and join given, for every $\Cal T_1, \Cal T_2 \in \Cal L_t(A)$, by
    $$
    \Cal T_1 \wedge \Cal T_2=\Cal T_1 \cap\Cal T_2 \qquad \text{and} \qquad \Cal T_1 \vee \Cal T_2 = \gen(\Cal T_1 \cup \Cal T_2).
    $$
\end{Prop}

\begin{proof}
    It immediately follows from the above discussion.
\end{proof}

Since torsion classes are in particular pretorsion classes, $\tors A$ is a subposet (but in general not a sublattice) of $\Cal L_t(A)$. Moreover, it is clear that $\gen (\Cal S)$ is contained in $\Cal T (\Cal S)$ for every class $\Cal S$ of modules in $\mod A$.

\begin{remark}
Note that in the above none of the assumptions on the category were used, so the lattice of pretorsion classes can be defined in the same way for any category $\Cal C$, substituting $\gen(\Cal T_1 \cup \Cal T_2)$ with the smallest pretorsion class containing both $\Cal T_1$ and $\Cal T_2$, and the results above hold. 
\end{remark}

For a class of modules $\Cal{X}\subseteq\Cal{C}$, we denote by $\add(\Cal{X})$ the (full) subcategory of $\Cal{C}$ whose modules are isomorphic to finite direct sums of modules of $\Cal X$.

\begin{Lemma}\label{lemma_gen_tor}
    Let $\Cal T$ be a pretorsion class of $\mod A$. Then, $\Cal T=\add(\Cal X)$ for a class $\Cal X$ of indecomposable modules of $\mod A$. In particular, for each class of modules $\Cal S$ of $\mod A$, there exists a class $\Cal X$ of indecomposable modules such that $\gen(\Cal S)=\add(\Cal X)$.
\end{Lemma}

\begin{proof}
    Since $\Cal T$ is a pretorsion class, it is closed under quotients, hence under direct summands. Therefore, it suffices to define $\Cal X$ to be the class of indecomposable modules of $\mod A$ that belong to $\Cal T$.
\end{proof}

\begin{remark}\label{rem_closure-of-add}
\begin{itemize}
    \item[(i)] Let $M \in \mod A$. Then, $\gen( M)=\add(\Cal X_M)$, where $\Cal X_M$ is the class of all indecomposable modules in $\mod A$ such that there exist $n \in \mathbb N$ and an epimorphisms $M^n \to N$. Equivalently, $\gen( M )$ is the class of all modules in $\mod A$ such that there exist $n \in \mathbb N$ and an epimorphisms $M^n \to N$, that is, the closure under quotients of $\add(M)$.
    
    \item[(ii)]  More generally, if $\Cal X$ is a class of (indecomposable) modules in $\mod A$, then $\gen( \Cal X )$ is the closure under quotients of $\add(\Cal X)$, that is, $N \in \gen( \Cal X)$ if and only if there exist (not necessarily distinct) $X_1, \dots, X_m \in \Cal X$ and an epimorphism $\oplus_{i=1}^m X_i \to N$.
    \end{itemize}
\end{remark}

\subsection{The dual: the lattice \texorpdfstring{$\Cal L_{tf}(A)$}{LtfA} of pretorsion-free classes}
It is natural to ask for similar results for the lattice of pretorsion-free classes, taking the ``dual" of the results stated above. First, notice that a subcategory $\Cal F$ of a category $\Cal C$ is pretorsion-free if and only if $\Cal F$ is an epireflective subcategory of $\Cal C$ (dual of Lemma~\ref{lemma_monocor_pretorsion}). Moreover, given a finite-dimensional $k$-algebra $A$, the category $\mod A$ is also well-copowered (i.e. for every module $M \in \mod A$, the quotients of $M$ form a set) and Artinian (i.e. for every module $M \in \mod A$, any descending chain of submodules of $M$ stabilises), hence using the dual of Proposition~\ref{prop_noetherian}, pretorsion-free subcategories of $\mod A$ are precisely those subcategories closed under submodules and finite direct sums. We can therefore define, for every class of finitely generated $A$-modules $\Cal S$, the subcategory $\cogen (\Cal S)$ of $\mod A$ to be the smallest pretorsion-free class containing  $\Cal S$. Finally, we can consider the poset $\Cal L_{tf}(A)$ of pretorsion-free classes of $\mod A$, ordered under inclusion, for which we have the following result.

\medskip
\begin{Prop}[{dual of Proposition~\ref{prop_complete_lattice_torsionclasses}}]~\label{prop_dual} 
The poset $\Cal L_{tf}(A)$ is a complete lattice, with meet and join given, for every $\Cal F_1, \Cal F_2 \in \Cal L_{tf}(A)$, by
    $$
    \Cal F_1 \wedge \Cal F_2=\Cal F_1 \cap\Cal F_2 \qquad \text{and} \qquad \Cal F_1 \vee \Cal F_2 = \cogen(\Cal F_1 \cup \Cal F_2).
    $$
\end{Prop}

Nevertheless, dualising the results presented in the rest of the paper is not straightforward. For this reason and for the sake of brevity, from now on we focus our attention only on pretorsion classes, leaving the study of pretorsion-free classes for future works.

\subsection{Completely join-irreducible elements in \texorpdfstring{$\Cal L_t(A)$}{Lt(A)}}\label{section_joinirred}

The aim of this section is to describe the completely join-irreducible elements in the lattice $\Cal L_t(A)$ and to compare them to those in the lattice of classic torsion classes $\tors A$.

We start by showing a key property of indecomposable modules. We point out that the second assertion in the following lemma is well-known and already noted in \cite[Lemma~1.12]{BHM}.

\begin{lemma}\label{lemma_indec_ess_gen}
    Let $M \in \mod A$ be an indecomposable module. If $\gen( M )=\gen(\Cal X)$ for a set $\Cal X$ of indecomposable modules, then $M\cong M'$ for some $M' \in \Cal X$. In particular, for $M$ and $N$ indecomposable modules in $\mod A$, $\gen( M)=\gen( N)$ implies $M\cong N$.
\end{lemma}

\begin{proof}
    By assumption, there is an object $N=N_1\oplus\dots\oplus N_t$, with $N_1,\dots,N_t$ indecomposable modules in $\Cal X$ and an epimorphism $\nu: N^n\to M$, for some integer $n\ge 1$. Moreover, as $N\in\gen (M)$, there is an epimorphism $\pi: M^m\to N$ for some integer $m\ge 1$.

    Consider the epimorphism $p:M^{mn}\to M$, obtained by composing the direct sum of $mn$ copies of $\nu$ with $\pi$. We have that $p=(p_1,\dots,p_{mn}):M^{mn}\to M$, where the components $p_1,\dots,p_{mn}$ are endomorphisms of $M$.  Assume for a contradiction that $p_i:M\to M$ is not an automorphism for each $i\in\{1,\dots,mn\}$, and consider the epimorphism
    $$
    f:=\begin{psmallmatrix}
        p \\ \vdots\\ p
    \end{psmallmatrix}: M^{mn}\to M^{mn},
    $$
    whose $mn$ components are copies of $p$. 

    By the Harada-Sai Lemma, see for example~\cite[Corollary~VI.1.3]{ARS}, composing $f$ with itself sufficiently many times gives the zero morphism (as all of its components are zero morphisms). This is a contradiction to the fact that $f$ is an epimorphism. Hence, for at least one $i\in\{1,\dots,mn\}$, we have that $p_i:M\to M$ is an automorphism. Since $p_i$ factors through $N$ by construction, it follows that $M$ is isomorphic to a direct summand $N_j$ of $N$. This completes the proof of the lemma as the last sentence in the statement is the special case when $\Cal X=\{N\}$ for $N$ indecomposable.
\end{proof}

\begin{theorem}\label{thm_joinirr}
   There is a bijection between the isomorphism classes of indecomposable modules in $\mod A$ and the completely join-irreducible elements of the lattice $\Cal L_t(A)$, given by $M\mapsto \gen(M).$
\end{theorem}

\begin{proof}
    Let $M$ be an indecomposable module and assume that $\gen( M )=\bigvee_{\Cal T \subsetneq \gen(M)} \Cal T$. Then there exist pretorsion classes $\Cal T_1,\dots, \Cal T_m \subsetneq \gen( M )$, positive integers $k_1,\dots,k_m$ and indecomposable modules $T_1^{(i)}, \dots, T_{k_i}^{(i)} \in \Cal T_i$ for every $i=1,\dots, m$ with an epimorphism
    $$
    \bigoplus_{i=1}^m \bigoplus_{j=1}^{k_i} T_j^{(i)} \to M.
    $$
    Therefore, $\gen( T_1^{(1)}, \dots, T_{k_1}^{(1)}, \dots, T_1^{(m)}, \dots, T_{k_m}^{(m)})=\gen( M )$ and by Lemma~\ref{lemma_indec_ess_gen}, we have that $M \cong T_j^{(\ell)}$ for some $\ell \in \{1,\dots,m\}$ and $j \in \{1, \dots, k_\ell\}$, which is a contradiction as $M \notin \Cal T_i$ for every $i=1, \dots, m$. Hence $\gen( M )$ is completely join-irreducible.

    Let now $\Cal T$ be a pretorsion class not of the form $\gen( M)$ for $M$ an indecomposable module. Suppose that $\Cal T$ is completely join-irreducible. Then, there exists an element $\Cal T_0 \subsetneq \Cal T$ such that for every $\Cal T' \in \Cal L_t(A)$, $\Cal T' \subsetneq \Cal T$ if and only if $\Cal T' \subseteq\Cal T_0$. As $\Cal T_0 \subsetneq \Cal T$, there exists an indecomposable module $N \in \Cal T \setminus \Cal T_0$, and $\gen( N ) \neq \Cal T$ by assumption. We therefore reach a contradiction since $\gen( N ) \nsubseteq \Cal T_0$ but $\gen( N ) \subsetneq \Cal T$. Thus $\Cal T$ is not completely join-irreducible.

    Finally, since if $M$ and $N$ are two non-isomorphic indecomposable modules then $\gen( M ) \neq \gen( N )$ by Lemma~\ref{lemma_indec_ess_gen}, the map is a bijection.
\end{proof}

The above theorem resembles previous results on the lattice $\tors A$ of (classic) torsion classes of $\mod A$ (see Section~\ref{subsection_torsA}): the completely join-irreducible elements in $\tors A$ are in bijection with the isomorphism classes of bricks in $\mod A$. Since bricks are special indecomposable modules, it is natural to ask when the completely join-irreducible elements of the two lattices coincide.

Recall that a module $M$ is {\em Gen-minimal} if whenever $M= M'\oplus M''$, then $M'\not\in\gen( M'' )$ \cite[Chapter~VI.6]{ASS}. Clearly, every indecomposable module is in particular Gen-minimal.
Recall also that a module $M$ is \emph{$\tau$-rigid} if $\Hom(M,\tau M)=0$, where $\tau$ is the Auslander-Reiten translation \cite{AIR}.

The following result is basically contained in \cite[Section~5]{AS}.

\begin{Prop}\label{prop_taurigid}
    Let $M$ be a Gen-minimal module in $\mod{A}$. Then $\gen( M)$ is a torsion class if and only if $M$ is a $\tau$-rigid module.
    In particular, if $A$ is a locally representation directed algebra, then $\gen( M)$ is a torsion class for every indecomposable module $M$ in $\mod{A}$.
\end{Prop}

\begin{proof}
    By \cite[Proposition~5.5 and Proposition~5.8]{AS}, we have that $M$ is $\tau$-rigid if and only if $\gen( M)$ is a torsion class (notice that in \cite[Section~5]{AS} the module $M$ is assumed to be Gen-minimal).
    
    The last sentence follows from the fact that if $A$ is a locally representation directed algebra, then every indecomposable module is $\tau$-rigid, see \cite{D91} and \cite[Proposition~4.1]{MP}.
\end{proof}

There are plenty of cases where some indecomposable modules in $\mod{A}$ are not $\tau$-rigid, and so $A$ is not locally representation directed. In these cases some indecomposable modules are not bricks, and so there are more join-irreducibles in the lattice of pretorsion classes than in the lattice of torsion classes of $\mod{A}$.
The following is an easy, small example.
\begin{example}\label{example_loop}
    Consider the quiver
    \[Q= \xymatrix{1 \ar@(ul,ur)^{\epsilon}}\]
    and the algebra $A=kQ/\Cal I$, where $\Cal I$ is the admissible ideal generated by $\epsilon^2$. Up to isomorphism, there are two indecomposable modules in $\mod{A}$, that is
    \[
    {\begin{smallmatrix}
        1
    \end{smallmatrix}}
    =\xymatrix{k \ar@(ul,ur)^{0} &\text{and}&
    {\begin{smallmatrix}
       1\\1 
    \end{smallmatrix}}=& k^2. \ar@(ul,ur)^{{\begin{psmallmatrix}
        0&0\\1&0
    \end{psmallmatrix}}}}
    \]
    Note that the first one is a brick but it is not $\tau$-rigid, in fact $\tau (\begin{smallmatrix}
        1
    \end{smallmatrix})=\begin{smallmatrix}
        1
    \end{smallmatrix}$ and so $\Hom (\begin{smallmatrix}
        1
    \end{smallmatrix}, \tau (\begin{smallmatrix}
        1
    \end{smallmatrix}))\neq 0$. On the other hand, the second one is not a brick, but it is $\tau$-rigid, in fact as it is projective we have that $\tau(\begin{smallmatrix}
       1\\1 
    \end{smallmatrix})=0$.
    There is only one join-irreducible in the lattice of torsion classes, that is $ \Cal T(\begin{smallmatrix}
        1
    \end{smallmatrix})= \mod{A}$, while there are two in the lattice of pretorsion classes, that is $\gen( \begin{smallmatrix}
        1
    \end{smallmatrix})=\add(\begin{smallmatrix}
        1
    \end{smallmatrix})$ and $\gen(\begin{smallmatrix}
        1\\1
    \end{smallmatrix})=\mod{A}$.
    As predicted by Proposition~\ref{prop_taurigid}, since $\begin{smallmatrix}
        1
    \end{smallmatrix}$ is not $\tau$-rigid, then $\gen( \begin{smallmatrix}
        1
    \end{smallmatrix})$ is not a torsion class.
\end{example}

Recall that the lattice of torsion classes is a subposet (but not a sublattice in general) of the lattice of pretorsion classes. We now show that locally representation directed algebras are exactly the algebras $A$ such that the completely join-irreducible elements of $\tors A$ coincide with the completely join-irreducible elements of $\Cal L_t(A)$.

\begin{theorem}\label{thm_join_irred_coincide}
    The completely join-irreducibles of $\Cal L_t(A)$ coincide with the completely join-irreducibles of the lattice $\tors A$ of torsion classes of $\mod{A}$ if and only if $A$ is a locally representation directed algebra.
\end{theorem}

\begin{proof}
    Suppose that $A$ is not a locally representation directed algebra, then there exists at least one indecomposable module $M$ in $\mod{A}$ that is not a brick. If $\gen (M)$ is not a torsion class or it is not completely join-irreducible in $\tors A$, then we are done. So we can assume that $\gen (M)=\Cal T(M)$ is completely join-irreducible in $\tors A$. By \cite[Theorem~3.3]{DIRRT} there exists a brick $B$ in $\mod A$ such that $\Cal T (M)=\Cal T(B)$ and $\gen (B) \neq \gen (M)$ (otherwise $B\cong M$ is a brick by Lemma~\ref{lemma_indec_ess_gen}). Then $\gen (B)\subsetneq \Cal T(B)$, hence $\gen (B)$ is a completely join-irreducible in $\Cal L_t(A)$ that is not an element (and hence not a completely join-irreducible) in $\tors A$.

    Suppose now that $A$ is a locally representation directed algebra.
    By Theorem~\ref{thm_joinirr}, the completely join-irreducibles in $\Cal L_t(A)$ are of the form $\gen( M)$ for $M$ indecomposable in $\mod A$ and $\gen( M)\neq \gen( N)$ if $M\not\cong N$.
    Moreover, by Proposition~\ref{prop_taurigid}, each such $\gen( M)$ is a torsion class and it is clearly the smallest torsion class containing $M$, that is, $\gen (M)=\Cal T(M)$. 
    The statement then follows by \cite[Theorem~3.3]{DIRRT}.
\end{proof}

\begin{corollary}\label{corollary_no_infinite_rep_type}
    If the lattice $\Cal L_t(A)$ is infinite, then its completely join-irreducible elements do not coincide with the completely join-irreducible elements of $\tors A$.    
\end{corollary}

\begin{proof}
    If the completely join-irreducibles of $\Cal L_t(A)$ coincide with the completely join-irreducibles of $\tors A$, then $A$ is a locally representation directed algebra, hence $A$ is representation-finite by \cite{D91}, and so both lattices $\Cal L_t(A)$ and $\tors A$ are finite.
\end{proof}

\section{Distributivity of \texorpdfstring{$\Cal L_t(A)$}{Lt(A)}}~\label{sec:dist}

The aim of this section is to characterise when the lattice $\Cal L_t(A)$ is distributive. We will show this happens exactly when $\Cal T_1 \vee \Cal T_2 = \gen( \Cal T_1 \cup \Cal T_2 )=\add(\Cal T_1 \cup \Cal T_2)$ for any pair of pretorsion classes $\Cal T_1$ and $\Cal T_2$ in $\Cal L_t(A)$.
Moreover, we give a full classification of the algebras $A$ such that $\Cal L_t(A)$ is a distributive lattice.

Note that, thanks to Proposition \ref{remark_directprod_distributive} below, in order to determine whether $\Cal L_t(A)$ is distributive, it is enough to study connected algebras $A$. Hence, in this section, we may assume $A$ is connected.

Recall that if $L_1$ and $L_2$ are two lattices, then their product $L_1\times L_2$ is the lattice whose underlying set is the cartesian product of $L_1$ and $L_2$ and the preorder is given, for every $(a,b),(c,d) \in L_1\times L_2$, by $(a,b)\leq (c,d)$ if and only if $a\leq c$ and $b\leq d$.

\begin{Prop}\label{remark_directprod_distributive}
    If $A=\prod_{i=1}^m A_i$ is a finite direct product of algebras, then
    $$
    \Cal L_t(A)\cong \Cal L_t(A_1)\times \dots \times \Cal L_t(A_m)
    $$
    as lattices. In particular, $\Cal L_t(A)$ is distributive if and only if $\Cal L_t(A_i)$ is distributive for each $i=1,\dots, m$. 
\end{Prop}

\begin{proof}
    Recall that $\mod A$ is categorically equivalent to $ \mod A_1\times \dots \times\mod A_m$ and hence $\mod A_i$ can be viewed as a subcategory of $\mod A$ in the obvious way.
    In particular, there are no non-zero maps between $\mod A_i$ and $\mod A_j$ for $i\neq j$.
    
    Thus, if $\Cal T$ is a pretorsion class in $\Cal L_t(A)$, denoting by $\Cal X_i$ the class of indecomposable modules in $\Cal T$ that belong to $\mod A_i$, we can write
    $$
    \Cal T= \gen (\Cal X_1\cup\dots \cup \Cal X_m)=\add\Big(\bigcup_{i=1}^m \gen (\Cal X_i)\Big).
    $$
    Thus $\Cal T$ corresponds to the product $\gen(\Cal X_1)\times\dots \times \gen (\Cal X_m)$, via the equivalence $\mod A \cong \mod A_1\times \dots \times\mod A_m$. Conversely, if $\Cal T_i \in \Cal L_t(A_i)$ for $i=1,\dots,m$, then the corresponding pretorsion class in $\mod A$ is given by $\add(\Cal T_1 \cup \dots \cup \Cal T_m)$. Hence $\Cal L_t(A)\cong \Cal L_t(A_1)\times \dots \times \Cal L_t(A_m)$.
 
    The last assertion of the statement follows from the easy to check fact that a product of lattices is distributive if and only if so are its factors.
\end{proof}

\subsection{A characterisation of the distributivity of \texorpdfstring{$\Cal L_t(A)$}{LtA}}

\begin{remark}\label{rem_jointly-epi}
    Notice that, for every pair of pretorsion classes $\Cal T_1,\Cal T_2$ in $\Cal L_t(A)$, $\add(\Cal T_1 \cup \Cal T_2) \subseteq \gen( \Cal T_1 \cup \Cal T_2 )$  and that we have the equality if and only if $\add(\Cal T_1 \cup \Cal T_2)$ is closed under quotients. The latter happens if and only if there are no epimorphisms $T_1\oplus T_2 \to Q$, with $T_i \in \Cal T_i$, for $i=1,2$, and $Q \notin \add(\Cal T_1 \cup \Cal T_2)$ (equivalently, if and only if there are no jointly epimorphisms $e_i \colon T_i \to Q$, with $T_i \in \Cal T_i$, for $i=1,2$, and $Q \notin \add(\Cal T_1 \cup \Cal T_2)$).

    Moreover, it is easy to see that since $\mod A$ is a Krull-Schmidt category, the equality $\add(\Cal T_1 \cup \Cal T_2) = \gen( \Cal T_1 \cup \Cal T_2 )$ for every pair of pretorsion classes $\Cal T_1,\Cal T_2 \in\Cal L_t(A)$ implies the equality $\add(\bigcup_{i=1}^n\Cal T_i) = \gen( \bigcup_{i=1}^n\Cal T_i )$ for a finite number of pretorsion classes $\Cal T_1, \dots, \Cal T_n \in \Cal L_t(A)$.
\end{remark}

\begin{lemma} \label{lemma_two-implies-infinite}
    Assume that $\add(\Cal T_1 \cup \Cal T_2) = \gen( \Cal T_1 \cup \Cal T_2 )$ for every pair of pretorsion classes $\Cal T_1$ and $\Cal T_2$. Then, $\add(\bigcup_{i \in I}\Cal T_i)=\gen( \bigcup_{i \in I}\Cal T_i )$ for every family $\{T_i\}_{i \in I}\subseteq \Cal L_t(A)$.
\end{lemma}

\begin{proof}
    Let $\Cal A:=\add(\bigcup_{i \in I}\Cal T_i) \subseteq \gen( \bigcup_{i \in I}\Cal T_i )$. In order to prove the statement, we only need to show that $\Cal A$ is closed under quotients. Consider an epimorphism $M \to N$ with $M \in \Cal A$. Then, there exist $i_1, \dots, i_m \in I$ and $M_{i_j} \in \Cal T_{i_j}$ for $j=1, \dots, m$ such that $M=M_{i_1}\oplus \cdots \oplus M_{i_m}$. Thus, $M \in \add(T_{i_1}\cup \dots \cup T_{i_m})=\gen( T_{i_1}\cup \dots \cup T_{i_m} )$ and therefore $N \in \add(T_{i_1}\cup \dots \cup T_{i_m}) \subseteq \Cal A$, as desired.
\end{proof}

\begin{theorem}\label{thm_distributive-characterisation}
    The lattice $\Cal L_t(A)$ of pretorsion classes is distributive if and only if $\add(\Cal T_1 \cup \Cal T_2) = \gen( \Cal T_1 \cup \Cal T_2 )$ for every pair of pretorsion classes $\Cal T_1$ and $\Cal T_2$ in $\Cal L_t (A)$.
\end{theorem}

\begin{proof}
    $(\Leftarrow)$ Let $\Cal T_1$, $\Cal T_2$ and $\Cal T_3$ be pretorsion classes of $\mod A$. In order to show that $\Cal L_t(A)$ is distributive, it is enough to show that
    $$\Cal T_1 \cap \add(\Cal T_2\cup \Cal T_3)=\add((\Cal T_1\cap \Cal T_2)\cup (\Cal T_1 \cap \Cal T_3)).$$
    Let $X\in \Cal T_1 \cap \add(\Cal T_2\cup \Cal T_3)$. As $\mod A$ is Krull-Schmidt, we can write 
    $$X\cong T_{2,1}\oplus \cdots \oplus T_{2,n}\oplus T_{3,1}\oplus \cdots \oplus T_{3,m},$$
    where $T_{2,i}\in \Cal T_1\cap \Cal T_2$ and $T_{3,i}\in \Cal T_1\cap \Cal T_3$. Hence $X\in \add((\Cal T_1\cap \Cal T_2)\cup (\Cal T_1 \cap \Cal T_3))$. The opposite inclusion follows by a similar argument and hence $\Cal L_t(A)$ is distributive.
    
    $(\Rightarrow)$ Assume that there exists a pair of pretorsion classes $\Cal T_1$ and $\Cal T_2$ in $\Cal L_t (A)$ such that $\add(\Cal T_1 \cup \Cal T_2) \subsetneq \gen( \Cal T_1 \cup \Cal T_2 )$. Then, there exists an indecomposable module $M \in \gen( \Cal T_1 \cup \Cal T_2 )\setminus \add(\Cal T_1 \cup \Cal T_2)$. Set $\Cal T:= \gen( M )$. We want to show that  
    $$\Cal T \cap (\Cal T_1 \vee \Cal T_2 )\neq \gen( (\Cal T \cap \Cal T_1) \cup (\Cal T \cap \Cal T_2) ).$$
    
    Clearly, $M \in \Cal T \cap (\Cal T_1 \vee \Cal T_2 )$, so it suffices to show that $M \notin \gen( (\Cal T \cap \Cal T_1) \cup (\Cal T \cap \Cal T_2) )$. First, observe that $M \notin (\Cal T \cap \Cal T_1) \cup (\Cal T \cap \Cal T_2)$ since $M \notin \add(\Cal T_1 \cup \Cal T_2)$. By Remark~\ref{rem_jointly-epi} it is enough to show that $M$ is not a quotient of a module $T_1 \oplus T_2$ for some $T_1 \in \Cal T \cap \Cal T_1$ and $T_2 \in \Cal T \cap \Cal T_2$. Assume the contrary, so that $\Cal T =\gen(T_1\oplus T_2)$. Note that $T_1\oplus T_2 $ cannot have $M$ as a direct summand as $M \notin \Cal T_1 \cup \Cal T_2$ and $\mod A$ is Krull-Schmidt, a contradiction by Lemma~\ref{lemma_indec_ess_gen}.
    \end{proof}

Recall that, for an $A$-module $M$, the \emph{radical of $M$}, denoted by $\rad M$, is the intersection of all the maximal submodules of $M$.
Recall also that by Nakayama's lemma (see for example \cite[Lemma~I.2.2]{ASS}), any non-zero finitely generated module $M$ over a finite-dimensional algebra has $\rad M\subsetneq M$, that is $M$ has at least one maximal submodule.

\begin{Prop}\label{prop_distributive_maxsubmodule}
    If every indecomposable module $M\in\mod{A}$ has exactly one maximal submodule, then the lattice $\Cal L_t(A)$ is distributive.
\end{Prop}
    \begin{proof}
    If every indecomposable module $M$ in $\mod{A}$ has a unique maximal submodule, then $\rad M$ is equal to such maximal submodule. If $f: N\to M$ is a morphism that is not an epimorphism, then $\im f\subseteq \rad M$.
    As a consequence, any morphism $N_1\oplus N_2\to M$, where each component is not an epimorphism has image in $\rad M \subsetneq M$ and it is not an epimorphism.
    Hence $\Cal L_t (A)$ is distributive by Remark~\ref{rem_jointly-epi} and Theorem~\ref{thm_distributive-characterisation}.
    \end{proof}

We point out that the converse of Proposition~\ref{prop_distributive_maxsubmodule} is also true, as we will show in Proposition~\ref{prop_iff_unique_max}.

A special case of the above are algebras $A$ such that each  indecomposable module $M$ in $\mod A$ is \emph{uniserial}, that is it has unique composition series, or equivalently its submodule lattice is a chain. Moreover, if we require $A$ to be basic and connected, then by \cite[Corollary~V.3.6]{ASS} every indecomposable module in $\mod A$ is uniserial if and only if $A$ is a \textit{Nakayama algebra}, that is $A\cong kQ/\Cal I$ for $Q$ either linearly oriented type $\mathbb{A}_n$ or an oriented $n$-cycle, for some positive integer $n$ and some admissible ideal $\Cal I$.

\begin{corollary}
    If $A$ is an algebra where each indecomposable module in $\mod A$ is uniserial, then $\Cal L_t(A)$ is distributive.
    In particular, if $A$ is a Nakayama algebra, then $\Cal L_t(A)$ is distributive.
\end{corollary}

\begin{proof}
    The first sentence is a direct consequence of Proposition~\ref{prop_distributive_maxsubmodule}. The second sentence follows from \cite[Corollary~V.3.6]{ASS}.
\end{proof}

Note that $\Cal L_t(A)$ can be distributive even if there are indecomposable modules in $\mod A$ that are not uniserial. For example, for $A=k(1\leftarrow 2\rightarrow 3)$, the indecomposable module 
$k\xleftarrow{1_k} k\xrightarrow{1_k} k$ is not uniserial, but it has a unique maximal submodule and $\Cal L_t (A)$ is distributive (see Theorem~\ref{thm_classification_distributive} and Figure~\ref{fig:A3_distributive2} in the appendix).

\subsection{A classification of distributivity of \texorpdfstring{$\Cal L_t(A)$}{LtA}}\label{section_classification_distr}
We now classify the bound quiver algebras $A=kQ/\Cal I$ for which $\mod A$ has a distributive lattice of pretorsion classes. Notice that working with bound quiver algebras is not restrictive, thanks to Remark~\ref{rem_equivalence_modquiv}.

The aim of this subsection is proving the following result.

\begin{theorem}\label{thm_classification_distributive}
    Let $A=kQ/\Cal I$ be a bound quiver algebra. Then $\Cal L_t(A)$ is distributive if and only if each vertex in $Q$ has at most one arrow ending at it and two going out of it, and when there are precisely one entering arrow $\alpha$ and two exiting arrows $\beta$ and $\gamma$, then $\alpha\beta$ or $\alpha\gamma$ is in $\Cal I$. In other words, $Q$ does not contain subquivers of the form

    \begin{itemize}\itemsep1em
    \item[(i)] a vertex $\bullet$ with two arrows ending at it:\\

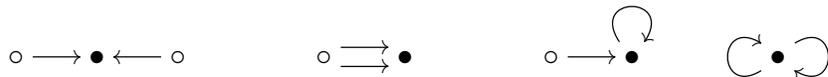
    \[\xymatrix@R=1em@C=1.5em{
    \circ \ar[r] & \bullet &\circ \ar[l] &&  \circ\ar@<-.7ex>[r] \ar@<.7ex>[r] & \bullet   &&   \circ \ar[r] & \bullet\ar@(ul,ur) &&
    \bullet \ar@(ur,dr)\ar@(dl,ul)
    }
    \]

    \item[(ii)] a vertex $\bullet$ with three arrows going out of it:\\

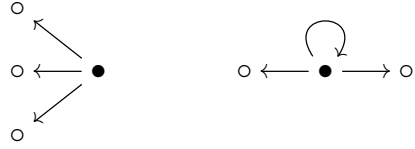
    \[\xymatrix@R=1em@C=1.5em{
     \circ&&&&&&&\\
     \circ & \bullet\ar[l]\ar[lu]\ar[ld]  &&\circ& \bullet\ar@(ul,ur) \ar[r]\ar[l] &\circ &&\\
     \circ &&&&&&&
    }\]

    \item[(iii)] a vertex $\bullet$ with two arrows going out and one arrow going in without appropriate relations, that is in the following  configurations neither $\alpha\beta$ nor $\alpha\gamma$ are in $\Cal I$, respectively neither $\epsilon^2$ nor $\epsilon\mu$ are in $\Cal I$:

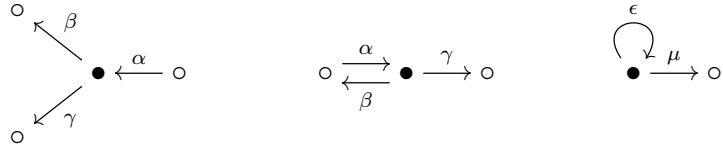
     \[\xymatrix@R=1em@C=1.5em{
       \circ &&&&&&&&&\\
     & \bullet \ar[lu]_\beta\ar[ld]^\gamma &  \circ\ar[l]_\alpha &&  \circ\ar@<.7ex>[r]^\alpha & \bullet \ar@<.7ex>[l]^\beta \ar[r]^\gamma &\circ && \bullet\ar@(ul,ur)^\epsilon \ar[r]^\mu &\circ\\
      \circ &&&&&&&&&
     }\]
\end{itemize}
\end{theorem}
Before proving the theorem, note that algebras appearing in the statement are a subclass of string algebras. We recall the definition below, see~\cite{BR} for more details.

\begin{definition}
    Let $Q$ be a quiver. An algebra $A=kQ/\mathcal{I}$ is called a \emph{string algebra} if
    \begin{itemize}
        \item[(a)] any vertex $v$ of $Q$ has at most two arrows coming in $v$ and at most two arrows going out of $v$,
        \item[(b)] given any arrow $\alpha\in Q$, there is at most one arrow $\beta$ with $\beta\alpha\notin \mathcal{I}$ and most one arrow $\gamma$ with $\alpha\gamma\notin \mathcal{I}$, 
        \item[(c)] all the relations are monomials of length at least two.
    \end{itemize}
\end{definition}

Quivers associated to the string algebras described in the statement of Theorem~\ref{thm_classification_distributive} have the following local configurations at a vertex $\bullet$:

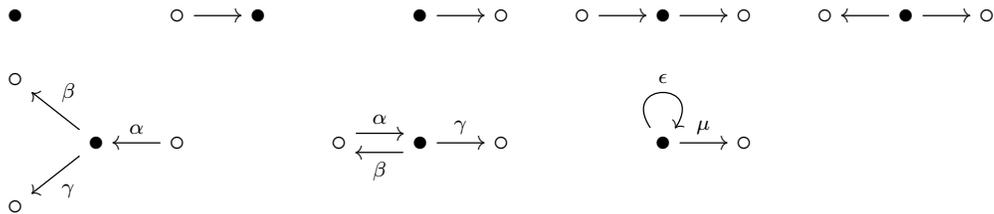
\[\xymatrix@R=1em@C=1.5em{
   \bullet &&\circ\ar[r]& \bullet  && \bullet \ar[r] &\circ& \circ\ar[r]& \bullet \ar[r] &\circ& \circ& \bullet \ar[l]\ar[r] &\circ 
\\
       \circ&&&&&&&&&\\
     & \bullet \ar[lu]_{\beta}\ar[ld]^{\gamma} &  \circ\ar[l]_{\alpha} &&\circ  \ar@<.7ex>[r]^\alpha & \bullet \ar@<.7ex>[l]^\beta\ar[r]^\gamma& \circ&& \bullet\ar@(ul,ur)^\epsilon \ar[r]^\mu &\circ\\
      \circ&&&&&&&&&
     }
     \]
     
     with relations $\alpha\beta$ or/and $\alpha\gamma\in\Cal I$;  $\epsilon^2$ or/and $\epsilon\mu\in\Cal I$ respectively. Note that it is allowed to have relations in the 4th configuration.

Representation theory of string algebras is well-understood~\cite{BR, CB}. For $A$ a string algebra, the indecomposable modules in $\mod A$ are divided into string and band modules, defined as follows. 

\begin{definition}
Let $A=kQ/\Cal I$ be a string algebra.
We write $Q=(Q_0,Q_1,s,t)$, where $Q_0$ is the set of vertices, $Q_1$ is the set of arrows, and $s,\ t \colon Q_1 \to Q_0$ are the source and target functions. For each $\alpha\in Q_1$, we define formally $\alpha^{-1}$ with $s(\alpha^{-1})=t(\alpha)$ and $t(\alpha^{-1})=s(\alpha)$. Let us denote by $Q_1^{-1}$ the set of these formal inverses of arrows. A \emph{string} (or a \emph{walk}) is a sequence $w=w_1w_2\cdots w_r$ for $w_i\in Q_1\cup Q_1^{-1}$ such that 
\begin{itemize}
    \item[(i)] $t(w_i)=s(w_{i+1})$,
    \item[(ii)] no subsequence of $w$ has the form $\alpha\alpha^{-1}$ or $\alpha^{-1}\alpha$,
    \item[(iii)] no subsequence $v=w_i w_{i+1}\dots w_{i+j}$ of $w$ is such that the corresponding path, nor the path corresponding to its formal inverse $v^{-1}:=w^{-1}_{i+j}\dots w^{-1}_{i+1}\dots w^{-1}_{i}$, is in $\Cal I$.
\end{itemize}

Trivial strings are denoted by $e_i$ for $i\in Q_0$. If $w =w_1w_2\cdots w_r$ with $w_i\in Q_1$ for every $1\leq i\leq n$, $w$ is called a \emph{direct string} and, dually, if $w^{-1}$ is a direct string, $w$ is called an \emph{inverse string}, see Figure~\ref{fig:direct-inverse}.

The \emph{bands} are strings $w$ starting and ending at the same vertex and such that $w^m$ is a string for any $m\ge 1$ and $w$ is not a power of a string of strictly smaller length. 

Each string $w$ defines an indecomposable \emph{string module} $M(w)$. As a representation, this is given by assigning a copy of the field $k$ to each of the vertices of the graph $w$ and a copy of the identity
map to each of the arrows in the graph. Note that $M(w)\cong M(v)$ if and only if $w=v^{-1}$ or $v=w^{-1}$. 

Moreover, each band $b$ (up to cyclic permutation and inversion) defines an infinite family of indecomposable \emph{band modules}.
\end{definition}

We can now see that string algebras  with underlying quiver satisfying the restrictions described in Theorem~\ref{thm_classification_distributive} are algebras of finite-representation type.
\begin{remark}\label{remark_distributive_finite}
Let $A=kQ/\Cal I$ be a connected string algebra as in the statement of Theorem~\ref{thm_classification_distributive}.
By construction, there is at most one cycle $C$ in $Q$ (otherwise there would be a vertex with two arrows ending at it), and $C^\ell\in \Cal I$ for some $\ell\ge 1$, as we assume $A$ to be finite-dimensional. Moreover, as we do not allow vertices with multiple arrows ending at them, there are no non-oriented cycles in $Q$. 

It is then straightforward to check that there are no bands and hence no band modules in $\mod A$. In other words, all indecomposable modules in $\mod A$ are string modules. We conclude that $A$ is of finite-representation type, see \cite[Lemma~II.8.1]{Erd}.
\end{remark}

In order to prove Theorem~\ref{thm_classification_distributive}, as a further preliminary step we show that in the presence of a subquiver of the form (i), (ii) or (iii), the lattice $\Cal{L}_t(A)$ is not distributive. In particular, the following examples show the key behaviours that make the lattice not distributive. In most of the following examples, we use the notation from~\cite[Remark 1.1]{S} for indecomposable modules.

\begin{example}\label{eg_nondistributive1}
    Let us consider the quiver $\xymatrix{\mathbb{A}_3: 1 \ar[r] & 2  & 3\ar[l]}$, the algebra $A=k\mathbb{A}_3$, and the indecomposable modules
    \[
    \xymatrix{
    {\begin{smallmatrix} 1\,\,3\\ 2 \end{smallmatrix}}: k\ar[r]^-1& k&k\ar[l]_1, & 
    {\begin{smallmatrix} 1\\ 2 \end{smallmatrix}}: k\ar[r]^-1& k&0,\ar[l] &
    {\begin{smallmatrix} 3\\ 2 \end{smallmatrix}}: 0\ar[r]& k&k.\ar[l]_1
    }
    \]
    We have that ${\begin{smallmatrix} 1\,\,3\\ 2 \end{smallmatrix}}$ is a quotient of ${\begin{smallmatrix} 1\\ 2 \end{smallmatrix}}\oplus {\begin{smallmatrix} 3\\ 2 \end{smallmatrix}}$, so ${\begin{smallmatrix} 1\,\,3\\ 2 \end{smallmatrix}}\in \gen({\begin{smallmatrix} 1\\ 2 \end{smallmatrix}})\vee \gen({\begin{smallmatrix} 3\\ 2 \end{smallmatrix}})$. Looking at dimension vectors of the representations, it is easy to see that ${\begin{smallmatrix} 1\,\,3\\ 2 \end{smallmatrix}}\not\in \gen ({\begin{smallmatrix} 1\\ 2 \end{smallmatrix}})$ and 
    ${\begin{smallmatrix} 1\,\,3\\ 2 \end{smallmatrix}}\not\in \gen ({\begin{smallmatrix} 3\\ 2 \end{smallmatrix}})$. Hence, by Remark~\ref{rem_jointly-epi} and Theorem~\ref{thm_distributive-characterisation}, we conclude that $\Cal L_t(A)$ is not distributive.
\end{example}

\begin{example}\label{eg_kronecker}
    Let $Q:\xymatrix{1 \ar@<-.7ex>[r] \ar@<.7ex>[r] & 2}$ be the Kronecker quiver and $A=kQ$. This is a very well studied example of $\mod kQ$, of infinite-representation type. The indecomposable modules and the morphisms between them are fully described, see for example~\cite[Chapter XI.4]{SS2}. In particular, for any $\lambda\in k$ there is an indecomposable module $\xymatrix{T_{1,\lambda}:  k \ar@<-.7ex>[r]_-{\lambda} \ar@<.7ex>[r]^-{1} & k}$, and different choices of $\lambda$ give non-isomorphic modules.

    For $\lambda\in k$ and $n\ge 1$, it is easy to check that the only morphisms of the form $(T_{1,\lambda})^n\to N_1$, with $N_1$ the indecomposable module as below, have the form
    \[\xymatrix{
(T_{1,\lambda})^n:\ar[d] && k^n \ar@<-.7ex>[r]_{\lambda\cdot I_n} \ar@<.7ex>[r]^{I_n} \ar[d]_{{\begin{psmallmatrix}
    a_1&\dots &a_n\\\lambda a_1&\dots &\lambda a_n
\end{psmallmatrix}}}& k^n\ar[d]^{{\begin{psmallmatrix}
    a_1&\dots &a_n
\end{psmallmatrix}}}\\
N_1 : && k^2 \ar@<-.7ex>[r]_{{\begin{psmallmatrix}
        0&1
    \end{psmallmatrix}}} \ar@<.7ex>[r]^{{\begin{psmallmatrix}
        1&0
    \end{psmallmatrix}}} & k,
}\]
where $a_1,\dots,a_n\in k$. These are clearly not surjective and so $N_1\notin \gen(T_{1,\lambda})$, for any choice of $\lambda\in k$.

On the other hand, we have the epimorphism

\[\xymatrix@R=3em{
T_{1,0}\oplus T_{1,1}: \ar[d] & k^2 \ar@<-.7ex>[r]_{{\begin{psmallmatrix}
    0 & 0 \\ 0 & 1
\end{psmallmatrix}}} \ar@<.7ex>[r]^{
{\begin{psmallmatrix}
    1 &0\\ 0&1
\end{psmallmatrix}}} \ar[d]_{{\begin{psmallmatrix}
    1 & 1\\ 0& 1
\end{psmallmatrix}}}& k^2\ar[d]^{{\begin{psmallmatrix}  1 & 1\end{psmallmatrix}}}
\\
N_1 : & k^2 \ar@<-.7ex>[r]_{{\begin{psmallmatrix}
        0&1
    \end{psmallmatrix}}} \ar@<.7ex>[r]^{{\begin{psmallmatrix}
        1&0
    \end{psmallmatrix}}} & k.
}\]

Hence $N_1\in \gen( T_{1,0} ) \vee \gen( T_{1,1} ) \neq \add(\gen( T_{1,0} )\cup \gen( T_{1,1} ))$ and the lattice $\Cal L_t(A)$ is not distributive by Remark~\ref{rem_jointly-epi} and Theorem~\ref{thm_distributive-characterisation}.
\end{example}

\begin{example}\label{eg_nondistributive4}
Consider the quiver 

\[\xymatrix{Q: &1 \ar[r]^\alpha & 2\ar@(ur,dr)^\epsilon}\]
and the algebra $A=kQ/\Cal I$ for some admissible ideal $\Cal I$, that is $\Cal I$ contains a relation of the form $\epsilon^r$ for some integer $r\geq 2$. Consider the indecomposable modules
\[
\xymatrix{
{\begin{smallmatrix}
    2\,\,1\\2
\end{smallmatrix}}: k \ar[r]^-{\begin{psmallmatrix}
    0\\1
\end{psmallmatrix}} & k^2\ar@(ur,dr)^{\begin{psmallmatrix}
    0&0\\1&0
\end{psmallmatrix}}
&& {\begin{smallmatrix}
    1\\2
\end{smallmatrix}}: k \ar[r]^-1 & k\ar@(ur,dr)^0
&&
{\begin{smallmatrix}
    2\\2
\end{smallmatrix}}: 0 \ar[r] & k^2\ar@(ur,dr)^{\begin{psmallmatrix}
    0&0\\1&0
\end{psmallmatrix}}
}
\]

Note that the above are indecomposable modules for any choice of $\Cal I$.
It is easy to see that $\begin{smallmatrix}
    2\,\,1\\2
\end{smallmatrix}$ is a quotient of $\begin{smallmatrix}
    1\\2
\end{smallmatrix}\oplus \begin{smallmatrix}
    2\\2
\end{smallmatrix}$ but $\begin{smallmatrix}
    2\,\,1\\2
\end{smallmatrix}\not\in\gen(\begin{smallmatrix}
    1\\2
\end{smallmatrix})$,
$\begin{smallmatrix}
    2\,\,1\\2
\end{smallmatrix}\not\in\gen(\begin{smallmatrix}
    2\\2
\end{smallmatrix})$. Hence the lattice $\Cal L_t(A)$ is not distributive by Remark~\ref{rem_jointly-epi} and Theorem~\ref{thm_distributive-characterisation}.
\end{example}

\begin{example}\label{eg_nondistributive5}
Consider the quiver
\[
\xymatrix{Q:&1 \ar@(ur,dr)\ar@(dl,ul)}
\]
    and the algebra $A=kQ/\Cal{I}$ for some admissible ideal $\Cal I$. Consider the indecomposable modules
    \[
    \xymatrix{
    M:&& k^2 \ar@(ur,dr)^0\ar@(dl,ul)^{{\begin{psmallmatrix}
        0&0\\1&0
    \end{psmallmatrix}}}
    &&
    N:& k^2 \ar@(ur,dr)^{{\begin{psmallmatrix}
        0&0\\1&0
    \end{psmallmatrix}}}\ar@(dl,ul)^0
    &&
    L:&& k^3 \ar@(ur,dr)^{{\begin{psmallmatrix}
        0&1&0\\0&0&0\\0&0&0
\end{psmallmatrix}}}\ar@(dl,ul)^{{\begin{psmallmatrix}
  0&0&1\\0&0&0\\0&0&0      
    \end{psmallmatrix}}}
    }
    \]
Note that the above are indecomposable modules for any choice of $\Cal I$.
It is routine to build an epimorphism from $M\oplus N$ to $L$, so $L\in \gen( M)\vee\gen( N)$. Moreover, it is easy to see that $L\notin \Cal \gen( M)$ and $L\notin \gen( N)$ and hence $L\notin \add(\gen( M)\cup \gen( N))\subsetneq \Cal \gen( M)\vee \gen( N)$.
Then, by Remark~\ref{rem_jointly-epi} and Theorem~\ref{thm_distributive-characterisation}, we conclude that $\Cal L_t(A)$ is not distributive.
\end{example}

\begin{example}\label{eg_nondistributive2}
    Let us consider the quiver 
    $$
    \xymatrix@R=1em{
     &1 
    \\
     \mathbb{D}_4: &2 & 4,\ar[l]\ar[lu]\ar[ld]
    \\
    &3
    }
    $$
    the algebra $A=k\mathbb{D}_4$,
    and the indecomposable modules

    \[
    \xymatrix@R=2em @C=2.5em{
     &k &&&0&&&k
    \\
     {\begin{smallmatrix}
         4\,\,4\\ 1\,2\,3
     \end{smallmatrix}}: &k & k^2,\ar[l]_{{\begin{psmallmatrix}
         1&1
     \end{psmallmatrix}}}\ar[lu]_{{\begin{psmallmatrix}
         1&0
     \end{psmallmatrix}}}\ar[ld]^{{\begin{psmallmatrix}
         0&1
     \end{psmallmatrix}}}
     &
     {\begin{smallmatrix}
        4\\2\,\,3
    \end{smallmatrix}}: & k & k,\ar[l]_{1}\ar[lu]\ar[ld]^{1}
    &
    {\begin{smallmatrix}
        4\\1\,\,3
    \end{smallmatrix}}: & 0 & k.\ar[l]\ar[lu]_1\ar[ld]^{1}
    \\
    &k&&&k&&&k
    }
    \]
    It is routine to build an epimorphism from  ${\begin{smallmatrix}
        4\\2\,\,3
    \end{smallmatrix}}\oplus {\begin{smallmatrix}
        4\\1\,\,3
    \end{smallmatrix}}$ to ${\begin{smallmatrix}
         4\,\,4\\ 1\,2\,3
     \end{smallmatrix}}$, so $\begin{smallmatrix}
         4\,\,4\\ 1\,2\,3
     \end{smallmatrix}\in \gen (\begin{smallmatrix}
        4\\2\,\,3
    \end{smallmatrix})\vee \gen(\begin{smallmatrix}
        4\\1\,\,3
    \end{smallmatrix})$. Looking at dimension vectors of the representations, it is easy to see that
    $\begin{smallmatrix}
         4\,\,4\\ 1\,2\,3
     \end{smallmatrix}\not\in \gen (\begin{smallmatrix}
        4\\2\,\,3
    \end{smallmatrix})$ and $\begin{smallmatrix}
         4\,\,4\\ 1\,2\,3
     \end{smallmatrix}\not\in \gen (\begin{smallmatrix}
        4\\1\,\,3
    \end{smallmatrix})$. Hence, by Remark~\ref{rem_jointly-epi} and Theorem~\ref{thm_distributive-characterisation}, we conclude that $\Cal L_t(A)$ is not distributive.
\end{example}

The following quiver gives a non-distributive lattice for similar reasons to Example~\ref{eg_nondistributive2} (a vertex with three arrows out of it).
\begin{example}\label{eg_nondistributive6}
    Consider the quiver 
    $Q:\xymatrix{2 & 1\ar@(ul,ur) \ar[r]\ar[l] & 3 }$
    and the algebra $A=kQ/\Cal I$ for some admissible ideal $\Cal I$. Consider the indecomposable modules
    \[
    \xymatrix@C=2.5em{
    {\begin{smallmatrix}
        1\\2\,\,3
    \end{smallmatrix}}: k & k\ar@(ul,ur)^0 \ar[r]^-1\ar[l]_-1 & k,
    \,\,\,\,\,\,\,
    {\begin{smallmatrix}
        1\\1\,\,2
    \end{smallmatrix}}: k & k^2\ar@(ul,ur)^{{\begin{psmallmatrix}
        0&0\\1&0
    \end{psmallmatrix}}} \ar[r]\ar[l]_-{{\begin{psmallmatrix}
        1&0
    \end{psmallmatrix}}} & 0,
    \,\,\,\,\,\,\,
    {\begin{smallmatrix}
        1\,\,1\\1\,2\,3
    \end{smallmatrix}}: k & k^3\ar@(ul,ur)^{{\begin{psmallmatrix}
        0&0&0\\1&0&0\\0&0&0
    \end{psmallmatrix}}} \ar[r]^{{\begin{psmallmatrix}
        0&0&1
    \end{psmallmatrix}}}\ar[l]_-{{\begin{psmallmatrix}
        1&0&1
    \end{psmallmatrix}}} & k.
    }
    \]
    Note that the above are indecomposable modules for any choice of $\Cal I$.
    It is routine to check that there is an epimorphism from  $\begin{smallmatrix}
        1\\2\,\,3
    \end{smallmatrix}\oplus \begin{smallmatrix}
        1\\1\,\,2
    \end{smallmatrix}$ to $\begin{smallmatrix}
        1\,\,1\\1\,2\,3
    \end{smallmatrix}$, so $\begin{smallmatrix}
        1\,\,1\\1\,2\,3
    \end{smallmatrix}\in \gen( \begin{smallmatrix}
        1\\2\,\,3
    \end{smallmatrix})\vee  \gen( \begin{smallmatrix}
        1\\1\,\,2
    \end{smallmatrix})$. Looking at dimension vectors of the representations, it is clear that $\begin{smallmatrix}
        1\,\,1\\1\,2\,3
    \end{smallmatrix}\notin \gen( \begin{smallmatrix}
        1\\1\,\,2
    \end{smallmatrix})$. Moreover, it is easy to check that $\begin{smallmatrix}
        1\,\,1\\1\,2\,3
    \end{smallmatrix}\notin \gen( \begin{smallmatrix}
        1\\2\,\,3
    \end{smallmatrix})$. Hence $\begin{smallmatrix}
        1\,\,1\\1\,2\,3
    \end{smallmatrix}\notin\add(\gen( \begin{smallmatrix}
        1\\1\,\,2
    \end{smallmatrix})\cup \gen( \begin{smallmatrix}
        1\\2\,\,3
    \end{smallmatrix}))$. By Remark~\ref{rem_jointly-epi} and Theorem~\ref{thm_distributive-characterisation}, we conclude that $\Cal L_t(A)$ is not distributive.
\end{example}

We now consider a different orientation of $\mathbb{D}_4$ than the one studied in a previous example.

\begin{example}\label{eg_nondistributive3}
    Let us consider the quiver
     $$
    
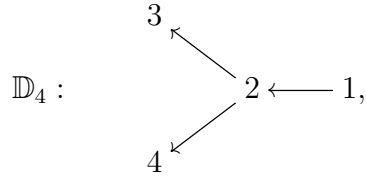
\xymatrix@R=1em{
   &3
    \\
   \mathbb{D}_4:&& 2 \ar[lu]\ar[ld] & 1, \ar[l]
    \\
    &4
    }
    $$
    the algebra $A=k\mathbb{D}_4$ and the indecomposable modules
    
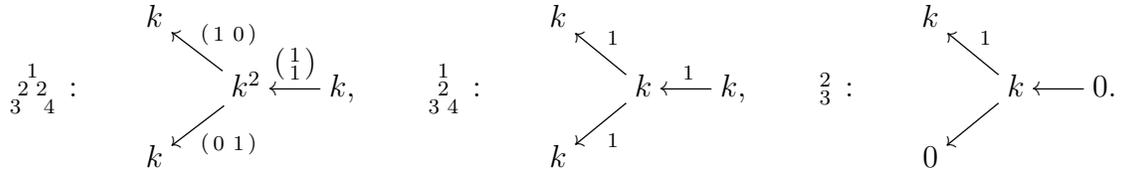
\[
    \xymatrix@R=0.5em @C=1.6em{
    &k&&&&k&&&&k
    \\
    {\begin{smallmatrix}
        1\\2\,\,2\\3\,\,\,\,\,\,4
    \end{smallmatrix}}: && k^2\ar[lu]_{{\begin{psmallmatrix}
        1&0
    \end{psmallmatrix}}}\ar[ld]^{{\begin{psmallmatrix}
        0&1
    \end{psmallmatrix}}}&k,\ar[l]_-{{\begin{psmallmatrix}
    1\\1
    \end{psmallmatrix}}}&
    {\begin{smallmatrix}
        1\\2\\3\,\,4
    \end{smallmatrix}}:&& k\ar[lu]_1\ar[ld]^1&k,\ar[l]_1
    &
    {\begin{smallmatrix}
        2\\3
    \end{smallmatrix}}:
     && k\ar[lu]_1\ar[ld]&0.\ar[l]
    \\
    &k&&&&k&&&&0
    }
    \]
    We have that $\begin{smallmatrix}
        1\\2\,\,2\\3\,\,\,\,\,\,4
    \end{smallmatrix}$ is a quotient of ${\begin{smallmatrix}
        1\\2\\3\,\,4
    \end{smallmatrix}}\oplus {\begin{smallmatrix}
        2\\3
    \end{smallmatrix}}$, so $\begin{smallmatrix}
        1\\2\,\,2\\3\,\,\,\,\,\,4
    \end{smallmatrix}\in \gen \Big({\begin{smallmatrix}
        1\\2\\3\,\,4
    \end{smallmatrix}}\Big)\vee \gen ({\begin{smallmatrix}
        2\\3
    \end{smallmatrix}})$. Looking at dimension vectors of the representations, it is clear that ${\begin{smallmatrix}
        1\\2\,\,2\\3\,\,\,\,\,\,4
    \end{smallmatrix}}\not\in\gen({\begin{smallmatrix}
        2\\3
    \end{smallmatrix}})$. Moreover, it is easy to directly check that there are no epimorphisms of the form $\Big(\begin{smallmatrix}
        1\\2\\3\,\,4
    \end{smallmatrix}\Big)^r\to \begin{smallmatrix}
        1\\2\,\,2\\3\,\,\,\,\,\,4
    \end{smallmatrix}$ for $r\ge 1$ and so $\begin{smallmatrix}
        1\\2\,\,2\\3\,\,\,\,\,\,4
    \end{smallmatrix}\not\in\gen\Big(\begin{smallmatrix}
        1\\2\\3\,\,4
    \end{smallmatrix}\Big)$. Hence, by Remark~\ref{rem_jointly-epi} and Theorem~\ref{thm_distributive-characterisation}, we conclude that $\Cal L_t(A)$ is not distributive.
\end{example}

The next case gives a non-distributive lattice for similar reasons to Example~\ref{eg_nondistributive3}.
\begin{example}\label{eg_nondistributive7}
    Consider the quiver $Q:\xymatrix{1 \ar@<.7ex>[r]^\alpha & 2 \ar@<.7ex>[l]^\beta\ar[r]^\gamma&3}$ and the algebra $A=kQ/\Cal I$ for some admissible ideal $\Cal I$ such that neither $\alpha\beta$ nor $\alpha\gamma$ is in $\Cal I$. Consider the indecomposable modules
    \[
    \xymatrix{
     {\begin{smallmatrix}
        2\\3
    \end{smallmatrix}}:
    0\ar@<.7ex>[r] & k \ar@<.7ex>[l]\ar[r]^1&k, &
    {\begin{smallmatrix}
        1\\2\\1\,\,3
    \end{smallmatrix}}:
    k^2\ar@<.7ex>[r]^-{{\begin{psmallmatrix}
        1&0
    \end{psmallmatrix}}} & k \ar@<.7ex>[l]^-{{\begin{psmallmatrix}
        0\\1
    \end{psmallmatrix}}}\ar[r]^1&k, 
    &
    {\begin{smallmatrix}
        \,\,1\\
        \,2\,\,2\\
        1\,\,\,\,\,\,3
    \end{smallmatrix}}:
    k^2\ar@<.7ex>[r]^-{{\begin{psmallmatrix}
        1&0\\1&0
    \end{psmallmatrix}}} & k^2 \ar@<.7ex>[l]^-{{\begin{psmallmatrix}
        0&0\\1&0
\end{psmallmatrix}}}\ar[r]^{{\begin{psmallmatrix}
        0&1
    \end{psmallmatrix}}}&k.
    }
    \]
    Note that the above are indecomposable modules in $\mod A$ thanks to the restriction we put on $\Cal I$.
    We have that $\begin{smallmatrix}
        \,\,1\\
        \,2\,\,2\\
        1\,\,\,\,\,\,3
    \end{smallmatrix}$ is a quotient of $\begin{smallmatrix}
        1\\2\\1\,\,3
    \end{smallmatrix}\oplus \begin{smallmatrix}
        2\\3
    \end{smallmatrix}$ and so 
    $$\begin{smallmatrix}
        \,\,1\\
        \,2\,\,2\\
        1\,\,\,\,\,\,3
    \end{smallmatrix}\in \gen\Big(\begin{smallmatrix}
        1\\2\\1\,\,3
    \end{smallmatrix}\Big)\vee \gen(\begin{smallmatrix}
        2\\3
    \end{smallmatrix}).$$
    Looking at dimension vectors of the representations, it is clear that $\begin{smallmatrix}
        \,\,1\\
        \,2\,\,2\\
        1\,\,\,\,\,\,3
    \end{smallmatrix}$  is not in $\gen( \begin{smallmatrix}
        2\\3
    \end{smallmatrix})$. Moreover, it is easy to directly check that $\begin{smallmatrix}
        \,\,1\\
        \,2\,\,2\\
        1\,\,\,\,\,\,3
    \end{smallmatrix}\not\in\gen\Big(\begin{smallmatrix}
       1\\2\\1\,\,3
    \end{smallmatrix}\Big)$.
    By Remark~\ref{rem_jointly-epi} and Theorem~\ref{thm_distributive-characterisation}, we conclude that $\Cal L_t(A)$ is not distributive.
\end{example}

The final case requires more computations involving indecomposable representations with bigger dimension vectors, but it still gives a non-distributive lattice.
\begin{example}\label{eg_nondistributive8}
    Consider the quiver $Q:\xymatrix{1\ar@(dl,ul)^\epsilon \ar[r]^\mu & 2}$ and the algebra $A=kQ/\Cal I$ for some admissible ideal $\Cal I$ such that $\epsilon^2$ and $\epsilon\mu$ are both not in $\Cal I$. Consider the indecomposable modules
    \[
    \xymatrix@C=3em{
    L={\begin{smallmatrix}
        1\\1\,\,2\\2\,\,\,\,\,\,\,\,
    \end{smallmatrix}}:\,\,\,\,& k^2\ar@(dl,ul)^{{\begin{psmallmatrix}
      0&0\\1&0  
    \end{psmallmatrix}}}\ar[r]^-{{\begin{psmallmatrix}
        1&0\\0&1
    \end{psmallmatrix}}} & k^2,
    \quad
    M={\begin{smallmatrix}
         \,\,1\\
        \,1\,\,1\\
        1\,\,\,\,\,\,2
    \end{smallmatrix}}:\,\,\,\,\,\,\,\,\,\,\,\,\,\,\,\,
    &
    k^4\ar@(dl,ul)^-{{\begin{psmallmatrix}
        0&0&0&0\\1&0&0&0\\0&1&0&0\\1&0&0&0
    \end{psmallmatrix}}} \ar[r]^-{{\begin{psmallmatrix}
        0&0&0&1
    \end{psmallmatrix}}} & k,
    \\
    \qquad
    N={\begin{smallmatrix}
        1&&&& 1\\
        &1&&1&&2\\
        2&&1
    \end{smallmatrix}}:&&
    k^5\ar@(dl,ul)^-{{\begin{psmallmatrix}
        0&1&0&0&1\\0&0&1&0&0\\0&0&0&0&0\\0&0&0&0&0\\0&0&0&1&0
    \end{psmallmatrix}}} \ar[r]^-{{\begin{psmallmatrix}
        0&1&0&0&0\\
        0&0&0&1&0
    \end{psmallmatrix}}} & k^2.
    }
    \]
    Note that the above are indecomposable modules in $\mod A$ thanks to the restriction we put on $\Cal I$.
    It is easy to check that $N$ is a quotient of $L\oplus M$, and so $N\in \gen(L)\vee\gen(M)$. Moreover, by direct computations, it is easy to check that $N\notin\gen(L)$ and $N\notin\gen(M)$.
    By Remark~\ref{rem_jointly-epi} and Theorem~\ref{thm_distributive-characterisation}, we conclude that $\Cal L_t(A)$ is not distributive. 
\end{example}

We are now ready to prove Theorem~\ref{thm_classification_distributive}.
\begin{proof}[Proof of Theorem~\ref{thm_classification_distributive}]
    Let $A=kQ/\Cal I$ be a bound quiver algebra.
    We first show that if $Q$ contains one of the subquivers listed in the statement of Theorem~\ref{thm_classification_distributive}, then $\Cal L_t (A)$ is not a distributive lattice.
    Note that it is enough to prove non-distributivity on the representations of the subquiver, as they correspond to the representations of $Q$ with $0$'s in the vertices outside the subquiver and $\gen$ of these representations is not affected by the rest of the quiver. By Examples~\ref{eg_nondistributive1}, \ref{eg_kronecker}, \ref{eg_nondistributive4}, \ref{eg_nondistributive5}, \ref{eg_nondistributive2}, \ref{eg_nondistributive6}, \ref{eg_nondistributive3},  \ref{eg_nondistributive7} and \ref{eg_nondistributive8}, we conclude that the lattice $\Cal L_t (A)$ is not distributive.
    
    Suppose now that $Q$ does not contain subquivers of type (i) and (ii), and if it contains subquivers of type (iii) then $\alpha\beta$ or $\alpha\gamma\in \Cal I$, $\epsilon^2$ or $\epsilon\mu\in \Cal I$, respectively.
    As seen in Remark~\ref{remark_distributive_finite}, $A$ is a string algebra of finite-representation type and all the indecomposable modules in $\mod A$ are string modules. Moreover, thanks to the restrictions, and especially as there is no vertex with two arrows ending at it and no double arrows in the quiver, each indecomposable module $M$ has corresponding string of one of the three forms:
    \begin{figure}[H]
     \scalebox{0.5}{
    \xymatrix@!{
    &&&& v_i\ar[ld]& v_i\ar[rd]&&&&&&& v_i\ar[ld]\ar[rd]\\
    &&&v_{i+1}\ar[ld]&&& v_{i+1}\ar[rd]&&&&& v_j\ar[ld]&& v_\ell\ar[rd]\\
    &&\iddots\ar[ld]&&&&& \ddots\ar[rd]&&& \iddots\ar[ld]&&&&\ddots\ar[rd]\\
    &v_{i+r}&&&&&&&v_{i+r}& v_{j+r}&&&&&&v_{\ell+s}
    }}\caption{Inverse string (left), Direct string (middle), other string (right).}~\label{fig:direct-inverse}
   \end{figure}
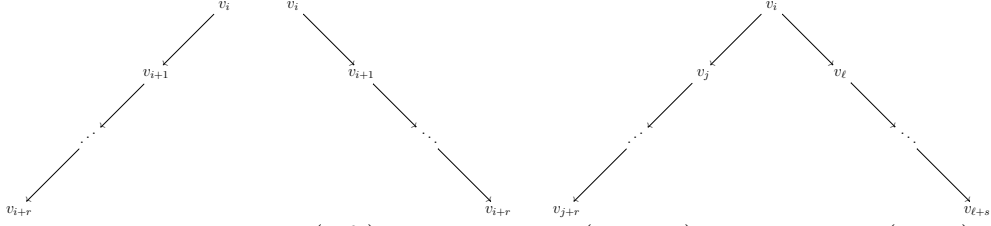
       
    for some $r,s\ge 1$ and $j\neq \ell$. Recall that $\rad A= R_Q/\Cal I$ where $R_Q$ is the arrow ideal of $kQ$, see \cite[Lemma~II.2.10]{ASS}. Then, as $\rad M = M \rad A$, it is immediate to see that in all three cases the top of $M$ is the simple module corresponding to the vertex $v_i$ and so $M$ has a unique maximal submodule (in the first two cases this is an indecomposable module, while in the third it is a direct sum of two indecomposable modules). Hence, each indecomposable module in $\mod A$ has a unique maximal submodule and by Proposition~\ref{prop_distributive_maxsubmodule} we conclude that $\Cal L_t(A)$ is distributive.
\end{proof}

Note that the proof of Theorem~\ref{thm_classification_distributive} also proves the converse of Proposition~\ref{prop_distributive_maxsubmodule}. More precisely, we have the following.

\begin{Prop}\label{prop_iff_unique_max}
     Let $A$ be a finite-dimensional $k$-algebra. Then, the lattice $\Cal L_t (A)$ is distributive if and only if each indecomposable module in $\mod A$ has a unique maximal submodule.
 \end{Prop}   

 \begin{proof}
     The if part follows from Proposition~\ref{prop_distributive_maxsubmodule}.
     For the other direction, first notice that if $\Cal L_t (A)$ is distributive, then $\mod A\cong \mod (kQ/\Cal I)$ for a bound quiver algebra $kQ/\Cal I$ as in the statement of Theorem~\ref{thm_classification_distributive}. As shown in the proof of the theorem, all indecomposable modules are string modules with a unique maximal submodule. The equivalence $\mod A\cong \mod (kQ/\Cal I)$ preserves the property that indecomposable modules have unique maximal submodules, see~\cite[Proposition~21.8]{AF}.
 \end{proof}

\begin{corollary}\label{cor_maximal_sub_boundquiver}
    Let $A$ be a finite-dimensional $k$-algebra. Then, $\Cal L_t(A)$ is a distributive lattice if and only if $\Cal L_t(A/I)$ is a distributive lattice for every ideal $I$ of $A$.
\end{corollary}

\begin{proof}
If $A$ is such that $\Cal L_t(A)$ is distributive, then by Proposition~\ref{prop_iff_unique_max} every indecomposable module in $\mod A$ has a unique maximal submodule. It easy to see that $A/I$, for $I$ an ideal of $A$, still has this property (notice that this remains true even if the algebra $A/I$ is not connected). 
Hence $\Cal L_t(A/I)$ is a distributive lattice for every ideal $I$ of $A$ by Proposition~\ref{prop_iff_unique_max}.
The other implication is clear taking $I=0$.
\end{proof} 

Notice that adding relations may result in obtaining a distributive lattice from a non-distributive one.

\begin{example}
    Let $\mathbb{D}_4$ be as in Example~\ref{eg_nondistributive3} and consider the algebra $A=k \mathbb{D}_4/\Cal I$, where $\Cal I$ is the admissible ideal generated by the path $1\to 2\to 3$ (or equivalently we could take $\Cal I$ to be generated by the path $1\to 2\to 4$).
    It is easy to see that every indecomposable module in $\mod A$ has exactly one maximal submodule. Hence by Proposition~\ref{prop_distributive_maxsubmodule} we conclude that the lattice $\Cal L_t(A)$ is distributive. On the other hand, as shown in Example~\ref{eg_nondistributive3}, $\Cal L_t(k\mathbb{D}_4)$ is not distributive.
\end{example}

\subsection{Distributive closure of the lattice of torsion classes}

Notice that $\tors A$ is a subposet but in general not a sublattice of $\Cal L_t(A)$. In this subsection we want to find a more precise relation between the two lattices when their completely join-irreducible elements coincide and $\Cal L_t(A)$ is distributive. Note that when $\Cal L_t(A)$ is distributive, it is a finite lattice by Theorem~\ref{thm_classification_distributive} and Remark~\ref{remark_distributive_finite}. Hence all of its join-irreducibles are completely join-irreducibles. Thus, in this subsection we will say join-irreducible and not completely join-irreducible.

An \emph{order ideal} $\mathrm{I}$ in a poset $(\Cal P,\leq)$ is a subset of $\Cal P$ such that if $x\in \mathrm{I}$ and  $y\leq x$ for $y\in \Cal P$, then $y\in \mathrm{I}$. Let us denote the set of order ideals of $\Cal P$ by $I(\Cal P)$. When $\Cal P$ is finite then the set $I(\Cal P)$ is actually a finite distributive lattice ordered by inclusion and every distributive lattice arises this way by Birkhoff's theorem of finite distributive lattices~\cite{B}.

For a finite lattice $\Cal L$, we denote by $Ji(\Cal L)$ the poset of join-irreducible elements of $\Cal L$ endowed with the partial order induced by the one in $\Cal L$. We then call the {\em distributive closure of $\Cal L$} the lattice $\Cal L^*:=I(Ji(\Cal L))$ of order ideals of the poset $Ji(\Cal L)$, ordered under inclusion. Clearly $\Cal L^*$ is a distributive lattice, being meet and join given by intersection and union respectively. Note that $\Cal L$ can be embedded as a subposet into $\Cal L^*$ via the assignment $\varepsilon\colon \Cal L \to \Cal L^*$ defined by
\[
\ell \mapsto \ell^* := \{x \in \Cal L \mid x\in Ji(\Cal L) \text{ and } x \leq \ell\}.
\] 

Thus, $\ell^*$ is the set of all join-irreducible elements of $\Cal L$ smaller or equal to $\ell$. It is clearly closed under $\leq$, hence it is an order-ideal of $Ji(\Cal L)$.

By Birkhoff's theorem, if $\Cal L$ is a finite distributive lattice, then $\varepsilon\colon \Cal L \to \Cal L^*$ is an isomorphism of lattices, whose inverse is the morphism $\varepsilon^{-1}\colon \Cal L^* \to \Cal L$ defined by $I\mapsto \vee I$ taking an order ideal to the join of its elements in $\Cal L$.

Moving back to the lattices $\tors A$ and $\Cal L_t(A)$, we get the following result.

\begin{lemma}\label{lemma_distr_closure_map}
    Let $A$ be a locally representation directed algebra and assume that $\Cal L_t(A)$ is distributive. Then, the distributive closure map $\varepsilon:\tors A\to (\tors A)^*$ factors as $\varepsilon=\varphi\iota$, where $\iota: \tors A\to \Cal L_t(A)$ is the inclusion map and $\varphi: \Cal L_t(A) \to (\tors A)^*$ is the lattice isomorphism defined by
    $$
    \Cal T \mapsto \varphi(\Cal T):=\{\gen (M) \in \Cal L_t(A)\mid M \text{ is indecomposable and } \gen(M)\subseteq \Cal T\}.
    $$
\end{lemma}

\begin{proof}
    Since $A$ is of finite-representation type, $\Cal L_t(A)$ is finite. Moreover, by Theorem~\ref{thm_join_irred_coincide}, the join-irreducible elements of $\tors A$ coincide with the join-irreducible elements of $\Cal L_t (A)$. The result then follows from what we have seen above.
\end{proof}

In the situation of the above lemma, we say that the distributive closure of $\tors A$ is identified with $\Cal L_t(A)$ (through the lattice isomorphism $\varphi$).

The following result is a consequence of Theorems~\ref{thm_classification_distributive} and \ref{thm_join_irred_coincide}.

\begin{Prop}\label{prop_distr_closure}
    Let $A=kQ/\Cal I$ be a bound quiver algebra. The following are equivalent.
    \begin{enumerate}
        \item $\Cal L_t (A)$ is finite (hence so is $\tors A$) and the distributive closure of $\tors A$ can be identified with $\Cal L_t (A)$;
        \item $A$ is a locally representation directed algebra in the class of string algebras described in Theorem~\ref{thm_classification_distributive};
        \item $Q$ is a quiver where each vertex has at most one arrow coming in and two arrows going out of it and when there are precisely one entering arrow $\alpha$ and two exiting arrows $\beta$ and $\gamma$, then $\alpha\beta$ or $\alpha\gamma$ is in $\Cal I$ and, in addition, $Q$ has no loops and $C\in\Cal I$ for every cycle $C$ in $Q$.
    \end{enumerate}
\end{Prop}

\begin{proof}
    Note that by \cite[Lemma~II.8.1]{Erd} the algebras described in Theorem~\ref{thm_classification_distributive} are all of finite-representation type. Hence, under the assumptions of (2) or (3) $\tors A$ and $\Cal L_t (A)$ are finite lattices. 

    The distributive closure of $\tors A$ can be identified with $\Cal L_t(A)$ if and only if the two lattices have the same join-irreducibles and $\Cal L_t(A)$ is distributive. By Theorems~\ref{thm_join_irred_coincide} and \ref{thm_classification_distributive}, this happens if and only if $A$ is a locally representation directed algebra in the class of string algebras described in Theorem~\ref{thm_classification_distributive}. Hence (1) and (2) are equivalent.

    Finally, we show the equivalence between (2) and (3). In other words, it remains to describe which algebras in our classification are locally representation directed.
    If $Q$ contains a cycle $C\not\in \Cal I$ of length $r\ge 1$, then there is a direct string $S$ with $s(S)=t(S)=v_i$ and a non-zero non-isomorphism of the corresponding string module $M(S)\to M(S)$ mapping the top copy of $v_i$ to the bottom copy of $v_i$ (and everything else to zero). Then $M(S)$ is an indecomposable module that is not a brick and $A$ is not locally representation directed.

    On the other hand, suppose that $Q$ contains no loops and $C\in\Cal I$ for every cycle $C$ in $Q$. Then each indecomposable module $M$ has corresponding string of one of the three forms in Figure~\ref{fig:direct-inverse},
    with no vertices repeating (so that each vertex has either a $0$ or $1$-dimensional $k$-vector space) and it is easy to see that $M$ is a brick.
\end{proof}

In the following example, we point out that the lattices $\tors A$ and $\Cal L_t(A)$ may both be distributive without coinciding. Moreover, it is also possible to have $\tors A$ distributive and $\Cal L_t(A)$ not distributive.
\begin{example}
\begin{itemize}
    \item Let $A=kQ/\langle\epsilon^2\rangle$ be as in Example~\ref{example_loop}, where
     \[Q= \xymatrix{1. \ar@(ul,ur)^{\epsilon}}\]
    Then $\tors A$ is distributive by Remark~\ref{rem_distributive_tors} and $\Cal L_t(A)$ is distributive by Theorem~\ref{thm_classification_distributive}. However, as seen in Example~\ref{example_loop} there is an indecomposable module which is not a brick and hence by Theorem~\ref{thm_join_irred_coincide} the join-irreducibles of the two lattices do not coincide, see Figure~\ref{fig:loop_example}.

    \begin{figure}[ht]
     \centerline{
        \xymatrix@C=1em{
        \Cal T({\begin{smallmatrix}
            1
        \end{smallmatrix}})=\mod A &&
        \gen({\begin{smallmatrix}
            1\\1
        \end{smallmatrix}})=\mod A
        \\
         && \gen({\begin{smallmatrix}\ar@{-}[u]
            1
        \end{smallmatrix}})=\add({{\begin{smallmatrix}
            1
        \end{smallmatrix}}})\\
        0\ar@{-}[uu]&&0\ar@{-}[u]
        }}
        \caption{On the left $\tors A$, on the right $\Cal L_t(A)$.}
        \label{fig:loop_example}
    \end{figure}
    
    \item Let now $A=kQ/\Cal I$ be a bound quiver algebra with underlying quiver
    \[
    \xymatrix{
    Q:&1 \ar@(ur,dr)\ar@(dl,ul)}
    \]
    then $\tors A$ is distributive and, for any choice of admissible ideal $\Cal I$, it is isomorphic to the lattice on the left of Figure~\ref{fig:loop_example} by Remark~\ref{rem_distributive_tors}. On the other hand, $\Cal L_t(A)$ is not distributive by Theorem~\ref{thm_classification_distributive} and, since $A$ is representation-infinite, it is an infinite lattice by Theorem~\ref{thm_joinirr}.
    \end{itemize}
\end{example}

\subsection{A different way of realising \texorpdfstring{$\Cal L_t(A)$}{LtA}}
When $\Cal L_t(A)$ is distributive, by Birkhoff's theorem, it is isomorphic to the lattice of order ideals of the poset of its join-irreducibles. In this subsection, we study this poset further, describing it in terms of the indecomposable projective modules in $\mod A$ and present a simple way of realising $\Cal L_t(A)$.

\begin{remark}\label{rem_max_epim}
    Let $M_1, \dots, M_n, N$ be $A$-modules where $N$ is indecomposable with a unique maximal submodule and let $f\colon M_1\oplus\cdots\oplus M_n \to N$ be an epimorphism. Then, there exists $k \in \{1,\dots, n\}$ such that $f\varepsilon_k\colon M_k \to N$ is an epimorphism (here, $\varepsilon_i \colon M_i\to M_1\oplus\cdots\oplus M_n$ denotes the canonical coprojection for every $i=1,\dots,n$). Indeed, if this is not the case, it means that the image of the morphisms $f \varepsilon_i$ is contained in $\rad(N)$ for every $i=1, \dots, n$, hence $\im(f)\subseteq \rad (N)$. Thus $f$ is not an epimorphism.
\end{remark}

\begin{lemma}
Assume that $\Cal L_t(A)$ is distributive. Then, for any indecomposable module $M$ there exists a unique (up to isomorphism) projective indecomposable module $P_M$ such that $\gen(M) \subseteq \gen(P_M)$.
\end{lemma}

\begin{proof}
Let $M \in \mod A$ be an indecomposable module and let $P_1,\dots, P_n$ be indecomposable projective modules such that $M$ is a quotient of $P_1\oplus \cdots\oplus P_n$. Set $T_i:=\gen(P_i)$ and notice that since $\Cal L_t(A)$ is distributive, $M \in \gen(\Cal T_1 \cup \dots \cup \Cal T_n )=\add(\Cal T_1 \cup \dots \cup \Cal T_n )$ by Theorem~\ref{thm_distributive-characterisation}, hence there exists $j \in \{1,\dots, n\}$ such that $M \in \Cal T_j=\gen(P_j)$. Now assume that $P$ and $P'$ are two projective indecomposable modules such that $M \in \gen(P)$ and $M \in \gen(P')$. Since $\Cal L_t(A)$ is distributive, $M$ has a unique maximal submodule by Proposition~\ref{prop_iff_unique_max}, hence by Remark~\ref{rem_max_epim}, there exist epimorphisms $P \to M$ and $P'\to M$. Thus, $P\cong P'$ by \cite[Lemma~5.1]{Lein}.
\end{proof}

For an indecomposable module $N \in \mod A$, set
    $$\gen(N)^* :=\{\gen(M)\mid M \text{ is indecomposable and } \gen(M)\subseteq \gen(N)\}.$$

\begin{Prop}
    Assume that $\Cal L_t(A)$ is distributive. Then, the set of join-irreducible elements of $\Cal L_t(A)$ is the disjoint union
    $$
    Ji( \Cal L_t(A) )= \coprod_{P} \gen(P)^*
    $$
    where $P$ runs over a set of representatives of the isomorphism classes of the indecomposable projective modules of $\mod A$.
\end{Prop}

\begin{proof}
    Immediately follows from the previous lemma.
\end{proof}

Let $\Cal E$ be the set of isomorphism classes of indecomposable modules of $\mod A$ endowed with the preorder induced by the epimorphisms, that is, for indecomposable modules $M, N \in \mod A$, $[M]_{\cong} \leq [N]_{\cong}$ if and only if there exists an epimorphism $N \to M$.

\begin{Prop}
Assume that $\Cal L_t(A)$ is distributive and consider the lattice $I(\Cal E)$ of order-ideals of $\Cal E$. Then, $I(\Cal E) \cong \Cal L_t(A)$.    
\end{Prop}

\begin{proof}
    First recall that by Proposition~\ref{prop_iff_unique_max}, since $\Cal L_t(A)$ is distributive, every indecomposable module $M$ has a unique maximal submodule. Thus, for indecomposable modules $M, N \in \mod A$, $[M]_{\cong} \leq [N]_{\cong}$ if and only if $\gen(M)\subseteq \gen(N)$, by Remark~\ref{rem_max_epim}. It follows that there is an isomorphism of posets $\Cal E \to Ji(\Cal L_t(A))$ given by $[M]_{\cong}\mapsto \gen(M)$, which extends to an isomorphism $I(\Cal E)\cong I(Ji(\Cal L_t(A)))=\Cal L_t(A)^*$. Since $\Cal L_t(A)$ is distributive, we get the conclusion.
\end{proof}

\appendix
\section{}

As an application of the importance of the lattices introduced in this paper, we show some ways they can be used to build pretorsion theories in $\mod A$. Moreover, we provide some explicit examples: we describe all the pretorsion theories in $\mod (k\mathbb{A}_2)$, and the lattices of pretorsion classes in $\mod (k\mathbb{A}_3)$ for the three possible orientations of $\mathbb{A}_3$.

\subsection{Building pretorsion theories using lattices}

In \cite[Theorems~3.1 and 4.2]{CF}, the first two authors showed how to build pretorsion theories using torsion theories in nice enough categories, such as $\mod A$. However, not all pretorsion theories can be built using the above mentioned results.

Lattices $\Cal L_t(A)$ of pretorsion classes and $\Cal L_{tf}(A)$ of pretorsion-free classes can be used to build all the remaining ones. In fact, all pretorsion theories have first half in $\Cal L_t(A)$ and second half in $\Cal L_{tf}(A)$. Recall that, unlike for classic torsion theories, the first half of a pretorsion theory does not determine its second half.
We show some methods to see which such pairs give a pretorsion theory and which can be easily discarded.

In the following, by $\Cal T^\perp$ we mean all objects $X\in\mod A$ such that $\Hom (T, X)=0$ for all $T\in\Cal T$. Dually, by $^\perp\Cal F$ we mean all objects $X\in\mod A$ such that $\Hom(X,F)=0$ for all $F\in\Cal F$.

\begin{Prop}\label{prop_build_pret_fromlattices}
  If $(\Cal T, \Cal F)$ is a pretorsion theory in $\mod A$, then $\Cal T=\gen (\Cal X)$  and $\Cal{F}=\cogen (\Cal Y)$, for two classes $\Cal X, \Cal Y$ of indecomposable modules.
  Moreover, 
  \begin{enumerate}
      \item if $\Cal T$ is a torsion class and $\Cal F$ is a torsion-free class, then $(\Cal T, \Cal F)$ is a pretorsion theory if and only if ${}^\perp\Cal F\subseteq \Cal T$ if and only if $\Cal T^\perp \subseteq \Cal F$;
      \item if $(\Cal T, \Cal F)$ is a pretorsion theory, then $\Cal T^\perp \subseteq \Cal F$ and ${}^\perp\Cal F\subseteq \Cal T$;
      \item if $\Cal T\in\Cal L_t(A)$, $\Cal{F}\in\Cal L_{tf}(A)$ and $\add (\Cal T\cup \Cal F)=\mod A$, then $(\Cal T, \Cal F)$ is a pretorsion theory. 
  \end{enumerate}
\end{Prop}
\begin{proof}
    The initial statement follows from Lemma~\ref{lemma_gen_tor} and its dual, while (1) follows from \cite[Theorem~3.1]{CF}.

    (2) If $(\Cal T, \Cal F)$ is a pretorsion theory, then by \cite[Proposition~2.7]{FFG} we have that $\Cal F=\Cal T^{\perp_{\Cal Z}}:=\{ M\in\mod A\mid \Hom (\Cal T,M)=\Triv(\Cal T, M)\}$ for $\Cal Z:=\Cal T\cap \Cal F$. Noting that $0\in \Cal Z$, we conclude that $\Cal T^\perp \subseteq \Cal F$. Dually, we have that ${}^\perp\Cal F\subseteq \Cal T$.

    (3) Assume that $\Cal T\in\Cal L_t(A)$, $\Cal{F}\in\Cal L_{tf}(A)$ and $\add (\Cal T\cup \Cal F)=\mod A$. Let $f:T\to F$ be a morphism with $T\in\Cal T$ and $F\in \Cal F$. As $\Cal F$ is epireflective, there exists an epimorphic $\Cal F$-envelope $g: T\to F_T$. Moreover, as $\Cal T$ is closed under quotients, it follows that $F_T\in \Cal Z:=\Cal T\cap \Cal F$. By definition of $\Cal F$-envelope, it follows that $f$ factors through $g$ and so $f$ is $\Cal Z$-trivial.
    
    Let $M\in \mod A$ be an indecomposable module. By assumption, $M\in\Cal T$ or $M\in \Cal F$. In the first case, we show that there is a short $\Cal Z$-exact sequence of the form
    \[M\xrightarrow{1_M} M\xrightarrow{g} Z_M\]
    with $Z_M\in \Cal Z$. As $M\in\Cal T$, it is clear that the identity map is a $\Cal T$-cover of $M$. Moreover, as $\Cal F$ is epireflective, there is an epimorphic $\Cal F$-envelope $g: M\to Z_M$, where by the above argument $Z_M\in \Cal Z$.
    The fact that $1_M$ is a $\Cal Z$-kernel of $g$ is clear.
    To see that $g$ is a $\Cal Z$-cokernel of $1_M$, let $h: M\to N$ be a $\Cal Z$-trivial morphism, that is $h=f\circ e$ for some $Z\in\Cal Z$ and morphisms $e:M\to Z$ and $f: Z\to N$. As $g$ is an $\Cal F$-envelope and $Z\in \Cal F$, it follows that $e$ factors through $g$. Hence $h$ factors through $g$ in a unique way, as $g$ is an epimorphism.
    By a dual argument, if $M\in \Cal F$, then there is a short $\Cal Z$-exact sequence of the form $Z_M\rightarrow M\xrightarrow{1_M} M$ with $Z_M\in \Cal Z$. Since $\mod A$ is Krull-Schmidt, it is easy to see that for a general $M\in\mod A$ there exists a short $\Cal Z$-exact sequence with middle term $M$. This completes the proof that $(\Cal T, \Cal F)$ is a pretorsion theory.
    \end{proof}

    \begin{remark}
        The above result shows how to use lattices of pretorsion and pretorsion-free classes to build pretorsion theories, significantly reducing the number of pairs that need a further check to determine if they form a pretorsion theory.
        \begin{itemize}
            \item When focusing on classic torsion and torsion-free classes, by (1) it is immediately clear which pairs give a pretorsion theory. 
            \item On the other hand, when $\Cal T\in \Cal L_t(A)$ is not a classic torsion class, to construct a pretorsion theory of the form $(\Cal T, \Cal F)$, by (2) it is enough to check $\Cal F$ in $\Cal L_{tf}(A)$ above a certain element, that is all the elements containing $\Cal T^\perp$ (which itself may not be an element in $\Cal L_{tf}(A)$). Among these, the elements $\Cal F\in \Cal L_{tf}(A)$ such that $\add (\Cal T\cup \Cal F)=\mod A$ automatically work by (3).
            \item The dual results show which pairs to exclude and which automatically work when fixing the second half of a pretorsion theory.
            \item In particular, by (3) we have that $(\Cal T, \mod A)$ and $(\mod A, \Cal F)$ are pretorsion theories for all $\Cal T\in\Cal L_t(A)$ and $\Cal{F}\in\Cal L_{tf}(A)$.
            \item Finally, note that a special case of pairs as in point (3) are {\em splitting torsion theories}, that is torsion theories such that each indecomposable object is either torsion or torsion-free or equivalently for each module $M$, the corresponding short exact sequence $0\to T_M\to M\to F_M\to 0$ splits. The latter have been widely studied, see for example \cite{Ho}.
        \end{itemize}
    \end{remark}
\subsection{Pretorsion theories in \texorpdfstring{$\mod k\mathbb{A}_2$}{mod kA2}}
Consider the quiver $Q=\mathbb{A}_2: 1\to 2$. The category $\mod k\mathbb{A}_2$ has three indecomposable objects up to isomorphism: the simple modules corresponding to the two vertices, namely $\begin{smallmatrix}
    1
\end{smallmatrix}$ and $\begin{smallmatrix}
    2
\end{smallmatrix}$, and the indecomposable projective-injective module
$\begin{smallmatrix}
    1\\2
\end{smallmatrix}$. Here, we are using the notation from \cite[Remark~1.1]{S} for indecomposable modules.

The lattices of pretorsion classes, $\Cal L_t(kQ)$, and of pretorsion-free classes $\Cal L_{tf}(kQ)$ of $\mod (kQ)$ are shown in Figure~\ref{fig:A2}. Notice that they are both distributive lattices.

Focusing on $\Cal L_t(kQ)$, as expected by Theorem~\ref{thm_joinirr}, the join-irreducible elements are $\gen$ of the indecomposable modules, that is $\gen(\begin{smallmatrix}
    1
\end{smallmatrix})=\add(\begin{smallmatrix}
    1
\end{smallmatrix})$, $\gen(\begin{smallmatrix}
    2
\end{smallmatrix})=\add(\begin{smallmatrix}
    2
\end{smallmatrix})$ and $\gen(\begin{smallmatrix}
    1\\2
\end{smallmatrix})=\add(\begin{smallmatrix}
    1
\end{smallmatrix},\begin{smallmatrix}
    1\\2
\end{smallmatrix} )$.

Moreover, as $kQ$ is a locally representation directed algebra, the distributive closure of $\tors (kQ)$ can be identified with $\Cal L_t(kQ)$ by Proposition~\ref{prop_distr_closure}. 
Observe that it is the lattice of order ideals of the poset

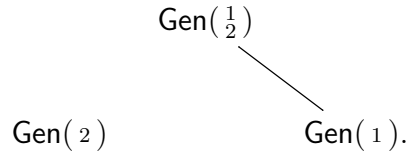
\[
\xymatrix@C=1em{
&\gen({\begin{smallmatrix}
    1\\2
\end{smallmatrix}})\\
\gen({\begin{smallmatrix}
    2
\end{smallmatrix}})
&&
\gen({\begin{smallmatrix}
    1
\end{smallmatrix}}).\ar@{-}[lu]
}
\]

Note that in both cases there is a single subcategory that is not a classic torsion, respectively torsion-free, class: $\add(\begin{smallmatrix}
    1
\end{smallmatrix},\begin{smallmatrix}
    2
\end{smallmatrix})$, as this is not closed under extensions. Moreover, $\tors (kQ)$ is a subposet but not a sublattice of $\Cal L_t(kQ)$, as a distributive lattice cannot have a non-distributive sublattice. The dual statement for lattices of torsion-free and pretorsion-free classes also holds.

Using Proposition~\ref{prop_build_pret_fromlattices}, one can check that there are $17$ pretorsion theories:
\begin{itemize}\itemsep0.5em
    \item There are $11$ pretorsion theories of the kind $(\Cal T, \mod k\mathbb{A}_2)$ or $(\mod k\mathbb{A}_2, \Cal F)$, for $\Cal T\in \Cal L_t(kQ)$ and $\Cal F\in \Cal L_{tf}(kQ)$, two of which are classic torsion theories.
    \item The remaining $3$ classic torsion theories:
    \[
(\add({\begin{smallmatrix}
    1
\end{smallmatrix}}), \add({\begin{smallmatrix}
    2
\end{smallmatrix}}, {{\begin{smallmatrix}
    1\\2
\end{smallmatrix}}})),
(\add({\begin{smallmatrix}
    2
\end{smallmatrix}}), \add({\begin{smallmatrix}
    1
\end{smallmatrix}})),
(\add({\begin{smallmatrix}
    1
\end{smallmatrix}}, {\begin{smallmatrix}
    1\\2
\end{smallmatrix}}), \add({\begin{smallmatrix}
    2
\end{smallmatrix}})).
    \]
\item The remaining one that can be found using Proposition~\ref{prop_build_pret_fromlattices}(1): 
\[
(\add({\begin{smallmatrix}
    1
\end{smallmatrix}}, {\begin{smallmatrix}
    1\\2
\end{smallmatrix}}), \add({\begin{smallmatrix}
    2
\end{smallmatrix}}, {\begin{smallmatrix}
    1\\2
\end{smallmatrix}})).
\]
\item The following two, that are pretorsion theories by Proposition~\ref{prop_build_pret_fromlattices}(3):
\[
(\add({\begin{smallmatrix}
    1
\end{smallmatrix}}, {\begin{smallmatrix}
    1\\2
\end{smallmatrix}}), \add({\begin{smallmatrix}
    1
\end{smallmatrix}}, {\begin{smallmatrix}
    2
\end{smallmatrix}})),
(\add({\begin{smallmatrix}
    1
\end{smallmatrix}}, {\begin{smallmatrix}
    2
\end{smallmatrix}}), \add({\begin{smallmatrix}
    2
\end{smallmatrix}}, {\begin{smallmatrix}
    1\\2
\end{smallmatrix}})).
\]
\end{itemize}

\begin{figure}[ht]
\scalebox{0.7}{
\xymatrix@!{
& *+[F.]{ \mod k\mathbb{A}_2 }
&&& *+[F.]{
%%%%%pretorsion-free
\mod k\mathbb{A}_2 }
\\
*+[F.]{ {\color{red}\add({\begin{smallmatrix}
    1
\end{smallmatrix}}, {{\begin{smallmatrix}
    1\\2
\end{smallmatrix}}})}} \ar@{-}[ru]
&&
\add({\begin{smallmatrix}
    1
\end{smallmatrix}}, {\begin{smallmatrix}
    2
\end{smallmatrix}})\ar@{-}[lu]
& *+[F.]{
%%%%pretorsion-free
{\color{red}\add({\begin{smallmatrix}
    2
\end{smallmatrix}}, {{\begin{smallmatrix}
    1\\2
\end{smallmatrix}}})}} \ar@{-}[ru]
&&
\add({\begin{smallmatrix}
    1
\end{smallmatrix}}, {\begin{smallmatrix}
    2
\end{smallmatrix}})\ar@{-}[lu]
\\
& *+[F.]{ {\color{red}\add({\begin{smallmatrix}
    1
\end{smallmatrix}})}} \ar@{-}[ru]\ar@{-}[lu]
&& *+[F.]{
{\color{red}\add({\begin{smallmatrix}
    2
\end{smallmatrix}})}} \ar@{-}[lu]
& *+[F.]{
%%%%%%pretorsion-free
{\color{red}\add({\begin{smallmatrix}
    2
\end{smallmatrix}})}}\ar@{-}[ru]\ar@{-}[lu]
&& *+[F.]{
{\color{red}\add({\begin{smallmatrix}
    1
\end{smallmatrix}})}} \ar@{-}[lu]
\\
&& *+[F.]{0} \ar@{-}[ru]\ar@{-}[lu]
&&& *+[F.]{0} \ar@{-}[ru]\ar@{-}[lu]
}
}
\caption{On the left, the lattice of pretorsion classes of $\mod k\mathbb{A}_2$, and on the right its lattice of pretorsion-free classes. In both cases, the join-irreducible elements are coloured in red. Dotted frames indicate the elements of the torsion and torsion-free classes.}
\label{fig:A2}
\end{figure}
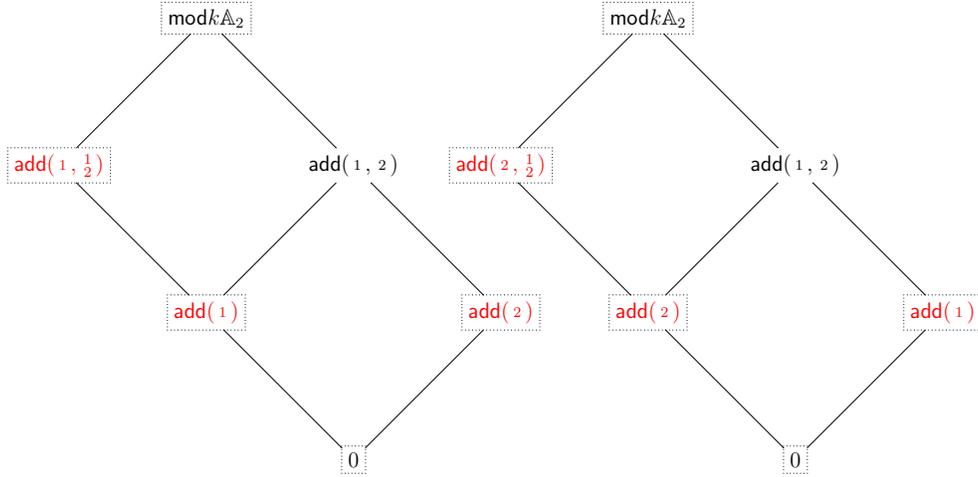

\subsection{Lattices of pretorsion classes for different orientations of \texorpdfstring{$\mathbb{A}_3$}{A3}}

Up to renumbering the vertices, there are three possible orientations of the graph

\[\mathbb{A}_3: \xymatrix{1\ar@{-}[r]&2\ar@{-}[r]&3.}\]

For two of them, the lattice of pretorsion classes $\Cal L_t(kQ)$ is distributive, while for one it is not. We illustrate below these three interesting cases. 

\medskip
\noindent (I) The linear orientation
\[\mathbb{A}_3: \xymatrix{1\ar[r]&2\ar[r]&3}\]
gives the distributive lattice $\Cal L_t (kQ)$ shown in Figure~\ref{fig:A3_linoriented}, where the join-irreducible elements are coloured in red and dotted frames indicate classic torsion classes. Observe that this lattice is isomorphic to the lattice of order ideals of the poset

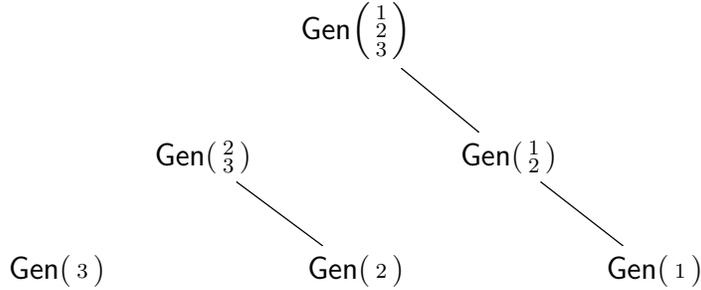
\[
\xymatrix@C=1em{
&&\gen\Big({\begin{smallmatrix}
    1\\2\\3
\end{smallmatrix}}\Big)\ar@{-}[rd]
\\
& \gen({\begin{smallmatrix}
    2\\3
\end{smallmatrix}})\ar@{-}[rd]
&&
\gen({\begin{smallmatrix}
    1\\2
\end{smallmatrix}})\ar@{-}[rd]
\\
\gen({\begin{smallmatrix}
    3
\end{smallmatrix}})
&&
\gen({\begin{smallmatrix}
    2
\end{smallmatrix}})
&&
\gen({\begin{smallmatrix}
    1
\end{smallmatrix}})
}
\]

and the distributive closure of the lattice $\tors (kQ)$ can be identified with $\Cal L_t(A)$ by Proposition~\ref{prop_distr_closure}.

As pointed out in Corollary~\ref{cor_maximal_sub_boundquiver}, $\Cal L_t(kQ/\Cal I)$, where $\Cal I$ is the ideal generated by the path of length $2$, is also a distributive lattice: it is the lattice consisting of the four lower cubes in Figure~\ref{fig:A3_linoriented} (i.e. the sublattice deleting all the subcategories containing the module $\begin{smallmatrix}
    1\\2\\3
\end{smallmatrix}$). Its join-irreducible elements are the 5 elements coloured in red in this sublattice.

\medskip
\noindent (II) The orientation

\[\mathbb{A}_3: \xymatrix{1\ar[r]&2&3\ar[l]}\]

gives the non-distributive lattice $\Cal L_t (kQ)$ shown in Figure~\ref{fig:A3_nondistributive}, where dotted frames indicate the classic torsion classes. Observe that the join-irreducible elements are the elements $\gen(M)$ for $M$ one of six indecomposable modules, coloured in red in the figure. Note that these are also the join-irreducible elements of the lattice $\tors (kQ)$.
Let us consider the poset of join-irreducible elements of $\Cal L_t (kQ)$. The distributive lattice of order ideals in this poset is a bigger lattice than $\Cal L_t (kQ)$, indeed it has two extra elements.

\medskip
\noindent (III) Finally, the orientation

\[\mathbb{A}_3: \xymatrix{1&2\ar[l]\ar[r]&3}\]

gives the distributive lattice $\Cal L_t (kQ)$ shown in Figure~\ref{fig:A3_distributive2}, where the join-irreducible elements are $\gen(M)$ for $M$ indecomposable module (coloured in red) and dotted frames indicate the classic torsion classes. Observe that this lattice is isomorphic to the lattice of order ideals of the the poset

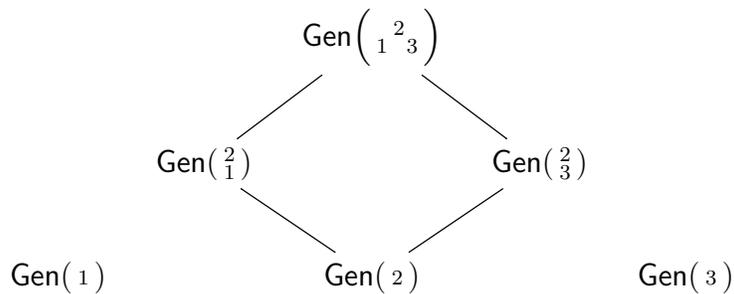
\[
\xymatrix@C=1em{
&&\gen\Big({\begin{smallmatrix}
    \,2\\
    1\,\,\,\,\,3
\end{smallmatrix}}
\Big)
\\
&\gen({\begin{smallmatrix}
    2\\1
\end{smallmatrix}})\ar@{-}[ru]
&&
\gen({\begin{smallmatrix}
    2\\3
\end{smallmatrix}})\ar@{-}[lu]
\\
\gen({\begin{smallmatrix}
    1
\end{smallmatrix}})
&&
\gen({\begin{smallmatrix}
    2
\end{smallmatrix}})\ar@{-}[ru]\ar@{-}[lu]
&&
\gen({\begin{smallmatrix}
    3
\end{smallmatrix}})
}
\]

and the distributive closure of the lattice $\tors (kQ)$ can be identified with $\Cal L_t(A)$ by Proposition~~\ref{prop_distr_closure}.

\begin{remark} 
\begin{enumerate}
    \item All three lattices of pretorsion classes for type $\mathbb{A}_3$ in Figures~\ref{fig:A3_linoriented}, \ref{fig:A3_nondistributive}, \ref{fig:A3_distributive2} have the same number of elements as the corresponding Coxeter group $S_4$. Furthermore, all three lattices embed into the lattice of Bruhat order on $S_4$.
This observation is also true for type $\mathbb{A}_2$, that is the lattice of pretorsion classes of $\mod \mathbb{A}_2$ has the same number of elements as the Coxeter group $S_3$ and it embeds into the lattice of Bruhat order on $S_3$. This observation hints mysterious combinatorial structures of lattices of pretorsion classes, that we are looking forward to unravel.
\item When lattices of pretorsion classes are not distributive, one can ask about \emph{semidistributivity}. However, $\Cal L_t(A)$ is not semidistributive in general. For instance, consider the lattice in Figure~\ref{fig:A3_nondistributive} and take the join-irreducible element   $\xymatrix@!{j:=
\add({\begin{smallmatrix}
    1\,\,\,\,3\\\,\,2
\end{smallmatrix}},
{\begin{smallmatrix}
    3
\end{smallmatrix}},  {\begin{smallmatrix}
    1
\end{smallmatrix}})}$. The set $\{y\mid y\wedge j = \widehat{j}\}$, where $\widehat{j}$ is the unique element $j$ covers, does not have a maximal element and this implies the fact that the lattice is not semidistributive~\cite{RST}.
\end{enumerate}
\end{remark}

\begin{figure}[ht]
\scalebox{0.5}{
\xymatrix@!{
&&& *+[F.]{\mod kQ}
\\
%%%%%%%%%%%%second level
&&\add({\begin{smallmatrix}
    3
\end{smallmatrix}}, {\begin{smallmatrix}
    2\\3
\end{smallmatrix}},{\begin{smallmatrix}
    2
\end{smallmatrix}}, {\begin{smallmatrix}
    1\\2
\end{smallmatrix}}, {\begin{smallmatrix}
    1
\end{smallmatrix}})\ar@{-}[ru]
&*+[F.]{
\add({\begin{smallmatrix}
    1\\2\\3
\end{smallmatrix}}, {\begin{smallmatrix}
    2\\3
\end{smallmatrix}},{\begin{smallmatrix}
    2
\end{smallmatrix}}, {\begin{smallmatrix}
    1\\2
\end{smallmatrix}}, {\begin{smallmatrix}
    1
\end{smallmatrix}})}\ar@{-}[u]
&
\add({\begin{smallmatrix}
    3
\end{smallmatrix}},
{\begin{smallmatrix}
    1\\2\\3
\end{smallmatrix}},{\begin{smallmatrix}
    2
\end{smallmatrix}}, {\begin{smallmatrix}
    1\\2
\end{smallmatrix}}, {\begin{smallmatrix}
    1
\end{smallmatrix}})\ar@{-}[lu]
\\
%%%%%%%%%%third level
& \add({\begin{smallmatrix}
    3
\end{smallmatrix}}, {\begin{smallmatrix}
    2\\3
\end{smallmatrix}},{\begin{smallmatrix}
    2
\end{smallmatrix}}, {\begin{smallmatrix}
    1
\end{smallmatrix}})\ar@{-}[ru]
&
\add({\begin{smallmatrix}
    2\\3
\end{smallmatrix}},{\begin{smallmatrix}
    2
\end{smallmatrix}}, {\begin{smallmatrix}
    1\\2
\end{smallmatrix}}, {\begin{smallmatrix}
    1
\end{smallmatrix}})\ar@{-}[u]\ar@{-}[ru]
&
\add({\begin{smallmatrix}
    3
\end{smallmatrix}},
{\begin{smallmatrix}
    2
\end{smallmatrix}}, {\begin{smallmatrix}
    1\\2
\end{smallmatrix}}, {\begin{smallmatrix}
    1
\end{smallmatrix}})\ar@{-}[lu]\ar@{-}[ru]
&*+[F.]{
\add(
{\begin{smallmatrix}
    1\\2\\3
\end{smallmatrix}},{\begin{smallmatrix}
    2
\end{smallmatrix}}, {\begin{smallmatrix}
    1\\2
\end{smallmatrix}}, {\begin{smallmatrix}
    1
\end{smallmatrix}})}\ar@{-}[lu]\ar@{-}[u]
& *+[F.]{
\add({\begin{smallmatrix}
    3
\end{smallmatrix}},
{\begin{smallmatrix}
    1\\2\\3
\end{smallmatrix}}, {\begin{smallmatrix}
    1\\2
\end{smallmatrix}}, {\begin{smallmatrix}
    1
\end{smallmatrix}})}\ar@{-}[lu]
\\ 
%%%%%%%%%%%fourth level
*+[F.]{\add({\begin{smallmatrix}
    3
\end{smallmatrix}}, {\begin{smallmatrix}
    2\\3
\end{smallmatrix}},{\begin{smallmatrix}
    2
\end{smallmatrix}})}\ar@{-}[ru]
&
\add({\begin{smallmatrix}
    2\\3
\end{smallmatrix}},{\begin{smallmatrix}
    2
\end{smallmatrix}}, {\begin{smallmatrix}
    1
\end{smallmatrix}})\ar@{-}[u]\ar@{-}[ru]
&
\add({\begin{smallmatrix}
    3
\end{smallmatrix}},
{\begin{smallmatrix}
    2
\end{smallmatrix}}, {\begin{smallmatrix}
    1
\end{smallmatrix}})\ar@{-}[lu]\ar@{-}[ru]
& *+[F.]{
\add(
{\begin{smallmatrix}
    2
\end{smallmatrix}}, {\begin{smallmatrix}
    1\\2
\end{smallmatrix}}, {\begin{smallmatrix}
    1
\end{smallmatrix}})}\ar@{-}[lu]\ar@{-}[u]\ar@{-}[ru]
&
\add({\begin{smallmatrix}
    3
\end{smallmatrix}},
 {\begin{smallmatrix}
    1\\2
\end{smallmatrix}}, {\begin{smallmatrix}
    1
\end{smallmatrix}})\ar@{-}[lu]\ar@{-}[ru]
& *+[F.]{
{\color{red} \add(
{\begin{smallmatrix}
    1\\2\\3
\end{smallmatrix}}, {\begin{smallmatrix}
    1\\2
\end{smallmatrix}}, {\begin{smallmatrix}
    1
\end{smallmatrix}})}}\ar@{-}[lu]\ar@{-}[u]
\\ 
%%%%%%%%%fifth level
*+[F.]{{\color{red}\add({\begin{smallmatrix}
    2\\3
\end{smallmatrix}},{\begin{smallmatrix}
    2
\end{smallmatrix}})}}\ar@{-}[ru]\ar@{-}[u]
&
\add({\begin{smallmatrix}
    3
\end{smallmatrix}},{\begin{smallmatrix}
    2
\end{smallmatrix}})\ar@{-}[lu]\ar@{-}[ru]
&
\add(
{\begin{smallmatrix}
    2
\end{smallmatrix}}, {\begin{smallmatrix}
    1
\end{smallmatrix}})\ar@{-}[lu]\ar@{-}[ru]\ar@{-}[u]
& *+[F.]{
\add(
{\begin{smallmatrix}
   3
\end{smallmatrix}}, {\begin{smallmatrix}
    1
\end{smallmatrix}})}\ar@{-}[lu]\ar@{-}[ru]
& *+[F.]{
{\color{red}\add(
 {\begin{smallmatrix}
    1\\2
\end{smallmatrix}}, {\begin{smallmatrix}
    1
\end{smallmatrix}})}} \ar@{-}[lu]\ar@{-}[ru]\ar@{-}[u]
\\
%%%%%%%%%sixth level
& *+[F.]{
{\color{red}\add({\begin{smallmatrix}
    2
\end{smallmatrix}})}} \ar@{-}[lu]\ar@{-}[ru]\ar@{-}[u]
& *+[F.]{
{\color{red}\add(
 {\begin{smallmatrix}
    3
\end{smallmatrix}})}} \ar@{-}[lu]\ar@{-}[ru]
&
*+[F.]{{\color{red}\add(
 {\begin{smallmatrix}
    1
\end{smallmatrix}})}} \ar@{-}[lu]\ar@{-}[ru]\ar@{-}[u]
\\
%%%%%%%%%bottom
&& *+[F.]{0}\ar@{-}[ru]\ar@{-}[lu]\ar@{-}[u]
}
}
\caption{The distributive lattice $\Cal L_t(kQ)$ for $Q: 1\rightarrow 2\rightarrow 3$. }
\label{fig:A3_linoriented}
\end{figure}
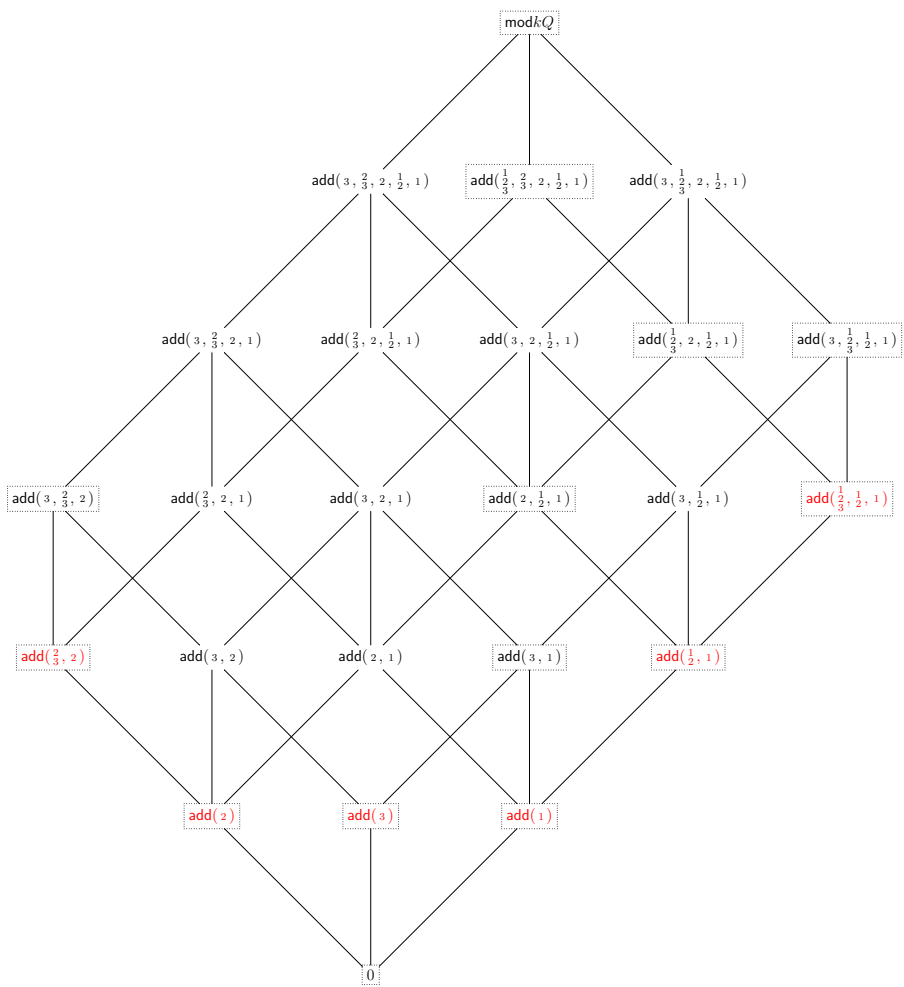

\begin{figure}[ht]
\scalebox{0.5}{
\xymatrix@!{
&&& *+[F.]{\mod kQ}
\\
%%%%%%%%%%%%second level
&&\add({\begin{smallmatrix}
    2
\end{smallmatrix}}, {\begin{smallmatrix}
    1\\2
\end{smallmatrix}},{\begin{smallmatrix}
    1\,\,\,\,3\\\,\,2
\end{smallmatrix}},
{\begin{smallmatrix}
    3
\end{smallmatrix}},  {\begin{smallmatrix}
    1
\end{smallmatrix}})\ar@{-}[ru]
& *+[F.]{
\add({\begin{smallmatrix}
    3\\2
\end{smallmatrix}}, {\begin{smallmatrix}
    1\\2
\end{smallmatrix}},{\begin{smallmatrix}
    1\,\,\,\,3\\\,\,2
\end{smallmatrix}},
{\begin{smallmatrix}
    3
\end{smallmatrix}},  {\begin{smallmatrix}
    1
\end{smallmatrix}})}\ar@{-}[u]
&
\add({\begin{smallmatrix}
    2
\end{smallmatrix}}, {\begin{smallmatrix}
    3\\2
\end{smallmatrix}},{\begin{smallmatrix}
    1\,\,\,\,3\\\,\,2
\end{smallmatrix}},
{\begin{smallmatrix}
    3
\end{smallmatrix}},  {\begin{smallmatrix}
    1
\end{smallmatrix}})\ar@{-}[lu]
\\
%%%%%%%%%%third level
& \add({\begin{smallmatrix}
    2
\end{smallmatrix}}, {\begin{smallmatrix}
    1\\2
\end{smallmatrix}},
{\begin{smallmatrix}
    3
\end{smallmatrix}},  {\begin{smallmatrix}
    1
\end{smallmatrix}})\ar@{-}[ru]
& *+[F.]{
\add({\begin{smallmatrix}
    1\\2
\end{smallmatrix}},{\begin{smallmatrix}
    1\,\,\,\,3\\\,\,2
\end{smallmatrix}},
{\begin{smallmatrix}
    3
\end{smallmatrix}},  {\begin{smallmatrix}
    1
\end{smallmatrix}})}\ar@{-}[u]\ar@{-}[ru]
&
\add({\begin{smallmatrix}
    2
\end{smallmatrix}},{\begin{smallmatrix}
    1\,\,\,\,3\\\,\,2
\end{smallmatrix}},
{\begin{smallmatrix}
    3
\end{smallmatrix}},  {\begin{smallmatrix}
    1
\end{smallmatrix}})\ar@{-}[lu]\ar@{-}[ru]
& *+[F.]{
\add( {\begin{smallmatrix}
    3\\2
\end{smallmatrix}},{\begin{smallmatrix}
    1\,\,\,\,3\\\,\,2
\end{smallmatrix}},
{\begin{smallmatrix}
    3
\end{smallmatrix}},  {\begin{smallmatrix}
    1
\end{smallmatrix}})}\ar@{-}[lu]\ar@{-}[u]
&
\add({\begin{smallmatrix}
    2
\end{smallmatrix}}, {\begin{smallmatrix}
    3\\2
\end{smallmatrix}},
{\begin{smallmatrix}
    3
\end{smallmatrix}},  {\begin{smallmatrix}
    1
\end{smallmatrix}})\ar@{-}[lu]
\\
%%%%%%%%%%%fourth level
*+[F.]{\add({\begin{smallmatrix}
    2
\end{smallmatrix}}, {\begin{smallmatrix}
    1\\2
\end{smallmatrix}}, {\begin{smallmatrix}
    1
\end{smallmatrix}})}\ar@{-}[ru]
&
\add( {\begin{smallmatrix}
    1\\2
\end{smallmatrix}},
{{\begin{smallmatrix}
    3
\end{smallmatrix}},  {\begin{smallmatrix}
    1
\end{smallmatrix}})}\ar@{-}[u]\ar@{-}[ru]
&
\add({\begin{smallmatrix}
    2
\end{smallmatrix}},
{\begin{smallmatrix}
    3
\end{smallmatrix}},  {\begin{smallmatrix}
    1
\end{smallmatrix}})\ar@{-}[lu]\ar@{-}[ru]\ar@{-}[rrru]
& *+[F.]{
{\color{red}\add({\begin{smallmatrix}
    1\,\,\,\,3\\\,\,2
\end{smallmatrix}},
{\begin{smallmatrix}
    3
\end{smallmatrix}},  {\begin{smallmatrix}
    1
\end{smallmatrix}})}}\ar@{-}[lu]\ar@{-}[u]\ar@{-}[ru]
& *+[F.]{
\add({\begin{smallmatrix}
    2
\end{smallmatrix}}, {\begin{smallmatrix}
    3\\2
\end{smallmatrix}},
{\begin{smallmatrix}
    3
\end{smallmatrix}})}\ar@{-}[lu]\ar@{-}[ru]
&
\add({\begin{smallmatrix}
    3\\2
\end{smallmatrix}},
{\begin{smallmatrix}
    3
\end{smallmatrix}},  {\begin{smallmatrix}
    1
\end{smallmatrix}})\ar@{-}[lu]\ar@{-}[u]
\\
%%%%%%%%%fifth level
*+[F.]{ {\color{red}\add({\begin{smallmatrix}
    1\\2
\end{smallmatrix}},{\begin{smallmatrix}
    1
\end{smallmatrix}})}}\ar@{-}[ru]\ar@{-}[u]
&
\add({\begin{smallmatrix}
    2
\end{smallmatrix}}, {\begin{smallmatrix}
    1
\end{smallmatrix}})\ar@{-}[lu]\ar@{-}[ru]
& *+[F.]{
\add({\begin{smallmatrix}
    3
\end{smallmatrix}},  {\begin{smallmatrix}
    1
\end{smallmatrix}})}\ar@{-}[lu]\ar@{-}[ru]\ar@{-}[u]\ar@{-}[rrru]
&
\add({\begin{smallmatrix}
    2
\end{smallmatrix}},{\begin{smallmatrix}
    3
\end{smallmatrix}})\ar@{-}[lu]\ar@{-}[ru]
& *+[F.]{
{\color{red}\add( {\begin{smallmatrix}
    3\\2
\end{smallmatrix}},{\begin{smallmatrix}
    3
\end{smallmatrix}})}}\ar@{-}[ru]\ar@{-}[u]
\\
%%%%%%%%%sixth level
& *+[F.]{
{\color{red}\add({\begin{smallmatrix}
    1
\end{smallmatrix}})}}\ar@{-}[lu]\ar@{-}[ru]\ar@{-}[u]
& *+[F.]{
{\color{red}\add(
 {\begin{smallmatrix}
    2
\end{smallmatrix}})}}\ar@{-}[lu]\ar@{-}[ru]
& *+[F.]{
{\color{red}\add(
 {\begin{smallmatrix}
    3
\end{smallmatrix}})}}\ar@{-}[lu]\ar@{-}[ru]\ar@{-}[u]
\\
%%%%%%%%%bottom
&& *+[F.]{0} \ar@{-}[ru]\ar@{-}[lu]\ar@{-}[u]
}
}
\caption{The non-distributive lattice $\Cal L_t(kQ)$ for $Q: 1\rightarrow 2\leftarrow 3$.}
\label{fig:A3_nondistributive}
\end{figure}

\medskip

\begin{figure}[ht]
\scalebox{0.435}{
\xymatrix@!{
&&& *+[F.]{\mod kQ}
\\
%%%%%%%%%%%%%second level
&& *+[F.]{\add(
 {\begin{smallmatrix}
    1
\end{smallmatrix}}, {\begin{smallmatrix}
    \,2\\
    1\,\,\,\,\,3
\end{smallmatrix}}, {\begin{smallmatrix}
    2\\3
\end{smallmatrix}}, {\begin{smallmatrix}
    2\\1
\end{smallmatrix}}, {\begin{smallmatrix}
    2
\end{smallmatrix}})}\ar@{-}[ru]
&
\add(
 {\begin{smallmatrix}
    1
\end{smallmatrix}}, {\begin{smallmatrix}
   3
\end{smallmatrix}}, {\begin{smallmatrix}
    2\\3
\end{smallmatrix}}, {\begin{smallmatrix}
    2\\1
\end{smallmatrix}}, {\begin{smallmatrix}
    2
\end{smallmatrix}})\ar@{-}[u]
& *+[F.]{
\add(
 {\begin{smallmatrix}
    3
\end{smallmatrix}}, {\begin{smallmatrix}
    \,2\\
    1\,\,\,\,\,3
\end{smallmatrix}}, {\begin{smallmatrix}
    2\\3
\end{smallmatrix}}, {\begin{smallmatrix}
    2\\1
\end{smallmatrix}}, {\begin{smallmatrix}
    2
\end{smallmatrix}})}\ar@{-}[lu]
\\
%%%%%%%%%%%%%third level
&&\add(
 {\begin{smallmatrix}
    1
\end{smallmatrix}},  {\begin{smallmatrix}
    2\\3
\end{smallmatrix}}, {\begin{smallmatrix}
    2\\1
\end{smallmatrix}}, {\begin{smallmatrix}
    2
\end{smallmatrix}})\ar@{-}[ru]\ar@{-}[u]
& *+[F.]{
{\color{red}\add({\begin{smallmatrix}
    \,2\\
    1\,\,\,\,\,3
\end{smallmatrix}},
  {\begin{smallmatrix}
    2\\3
\end{smallmatrix}}, {\begin{smallmatrix}
    2\\1
\end{smallmatrix}}, {\begin{smallmatrix}
    2
\end{smallmatrix}})}}\ar@{-}[lu]\ar@{-}[ru]
&
\add(
 {\begin{smallmatrix}
    3
\end{smallmatrix}},  {\begin{smallmatrix}
    2\\3
\end{smallmatrix}}, {\begin{smallmatrix}
    2\\1
\end{smallmatrix}}, {\begin{smallmatrix}
    2
\end{smallmatrix}})\ar@{-}[lu]\ar@{-}[u]
\\
%%%%%%%%%fourth level
&&& *+[F.]{
\add(
   {\begin{smallmatrix}
    2\\3
\end{smallmatrix}}, {\begin{smallmatrix}
    2\\1
\end{smallmatrix}}, {\begin{smallmatrix}
    2
\end{smallmatrix}})}\ar@{-}[lu]\ar@{-}[u]\ar@{-}[ru]
\\
%%%%%%%%%fifth level
&
\add(
 {\begin{smallmatrix}
    1
\end{smallmatrix}}, {\begin{smallmatrix}
    3
\end{smallmatrix}},  {\begin{smallmatrix}
    2\\1
\end{smallmatrix}}, {\begin{smallmatrix}
    2
\end{smallmatrix}})\ar@{-}[rruuu]
&&&&
\add(
 {\begin{smallmatrix}
    1
\end{smallmatrix}}, {\begin{smallmatrix}
    3
\end{smallmatrix}}, {\begin{smallmatrix}
    2\\3
\end{smallmatrix}},  {\begin{smallmatrix}
    2
\end{smallmatrix}})\ar@{-}[lluuu]
\\
%%%%%%%%sixth level
*+[F.]{\add(
 {\begin{smallmatrix}
    1
\end{smallmatrix}}, {\begin{smallmatrix}
    2\\1
\end{smallmatrix}}, {\begin{smallmatrix}
    2
\end{smallmatrix}})}\ar@{-}[ru]\ar@{-}[rruuu]
&&
\add(
 {\begin{smallmatrix}
    3
\end{smallmatrix}},  {\begin{smallmatrix}
    2\\1
\end{smallmatrix}}, {\begin{smallmatrix}
    2
\end{smallmatrix}})\ar@{-}[lu]\ar@{-}[rruuu]
&&
\add(
 {\begin{smallmatrix}
    1
\end{smallmatrix}},  {\begin{smallmatrix}
    2\\3
\end{smallmatrix}}, {\begin{smallmatrix}
    2
\end{smallmatrix}})\ar@{-}[ru]\ar@{-}[lluuu]
&& *+[F.]{
\add(
 {\begin{smallmatrix}
    3
\end{smallmatrix}},  {\begin{smallmatrix}
    2\\3
\end{smallmatrix}}, {\begin{smallmatrix}
    2
\end{smallmatrix}})}\ar@{-}[lu]\ar@{-}[lluuu]
\\
%%%%%%%seventh level
& *+[F.]{
{\color{red}\add(
 {\begin{smallmatrix}
    2\\1
\end{smallmatrix}}, {\begin{smallmatrix}
    2
\end{smallmatrix}})}}\ar@{-}[ru]\ar@{-}[lu]\ar@{-}[ru]\ar@{-}[rruuu]
&&&& *+[F.]{
{\color{red}\add(
 {\begin{smallmatrix}
    2\\3
\end{smallmatrix}}, {\begin{smallmatrix}
    2
\end{smallmatrix}})}}\ar@{-}[ru]\ar@{-}[lu]\ar@{-}[lluuu]
\\
%%%%%%%%eigth level
&&&
\add(
   {\begin{smallmatrix}
    1
\end{smallmatrix}}, {\begin{smallmatrix}
    3
\end{smallmatrix}}, {\begin{smallmatrix}
    2
\end{smallmatrix}})\ar@{-}[lluuu]\ar@{-}[rruuu]
\\
%%%%%%%%ninth level
&& 
\add(
   {\begin{smallmatrix}
    1
\end{smallmatrix}},  {\begin{smallmatrix}
    2
\end{smallmatrix}})\ar@{-}[lluuu]\ar@{-}[rruuu]\ar@{-}[ru]
& *+[F.]{
\add(
   {\begin{smallmatrix}
    1
\end{smallmatrix}}, {\begin{smallmatrix}
    3
\end{smallmatrix}})}\ar@{-}[u]
&
\add({\begin{smallmatrix}
    3
\end{smallmatrix}}, {\begin{smallmatrix}
    2
\end{smallmatrix}})\ar@{-}[lluuu]\ar@{-}[rruuu]\ar@{-}[lu]
\\
%%%%%%tenth level
&& *+[F.]{
{\color{red}\add(
   {\begin{smallmatrix}
    1
\end{smallmatrix}})}}\ar@{-}[u]\ar@{-}[ru]
& *+[F.]{
{\color{red}\add(
 {\begin{smallmatrix}
    2
\end{smallmatrix}})}}\ar@{-}[lluuu]\ar@{-}[rruuu]\ar@{-}[ru]\ar@{-}[lu]
& *+[F.]{
{\color{red}\add(
   {\begin{smallmatrix}
    3
\end{smallmatrix}})}}\ar@{-}[u]\ar@{-}[lu]
\\
%%%%%bottom
&&&
*+[F.]{0}\ar@{-}[lu]\ar@{-}[u]\ar@{-}[ru]
}
}
\caption{The distributive lattice $\Cal L_t(kQ)$ for $Q: 1\leftarrow 2\rightarrow 3$.}
\label{fig:A3_distributive2}
\end{figure}

\end{document}